\documentclass[10pt]{amsart}
\include{macro}

\newcommand{\grlf}{\operatorname{Gr_{l.f.}}}
\newcommand{\rk}{\operatorname{\underline{rank}}}

\title[Locally free Caldero--Chapoton functions via reflections]{Locally free Caldero--Chapoton functions via reflections}

\author{Lang Mou}
\address{Lang Mou\newline
Department of Pure Mathematics and Mathematical Statistics,\newline
University of Cambridge, CB3 0WB, United Kingdom}
\email{lmou.math@gmail.com}
\date{}

\begin{document}
\maketitle

\begin{abstract}
We study the reflections of locally free Caldero--Chapoton functions associated to representations of Gei\ss--Leclerc--Schr\"oer's quivers with relations for symmetrizable Cartan matrices. We prove that for rank 2 cluster algebras, non-initial cluster variables are expressed as locally free Caldero--Chapoton functions of locally free indecomposable rigid representations. Our method gives rise to a new proof of the locally free Caldero--Chapoton formulas obtained by Gei\ss--Leclerc--Schr\"oer in Dynkin cases. For general acyclic skew-symmetrizable cluster algebras, we prove the formula for any non-initial cluster variable obtained by almost sink and source mutations.
\end{abstract}

\tableofcontents

\section{Introduction}

Cluster algebras are invented by Fomin and Zelevinsky \cite{fomin2002cluster} in connection with dual canonical bases and total positivity. A cluster algebra $\mathcal A(B)$ associated to a skew-symmetrizable matrix $B$ is a subalgebra of $\mathbb Q(x_1, \dots, x_n)$ generated by a distinguished set of generators called \emph{cluster variables} obtained by certain iterations called \emph{mutations}. A first remarkable feature is that they turn out to be Laurent polynomials with integer coefficients. Much effort has been taken to give formulas or interpretations of these Laurent polynomials since the invention of cluster algebras.

The classification of finite type cluster algebras is identical to the Cartan--Killing classification of finite root systems \cite{fomin2003cluster}. In particular, non-initial cluster variables are naturally in bijection with positive roots of the corresponding root system. Meanwhile, Gabriel's theorem \cite{gabriel1972unzerlegbare} states that the indecomposable representations of a Dynkin quiver are in bijection with positive roots, thus further in bijection with non-initial cluster variables. Caldero and Chapoton \cite{caldero2006cluster} showed that any non-initial cluster variable can be obtained directly from its corresponding quiver representation as the generating function of Euler characteristics of quiver Grassmannians of subrepresentations, which we now call the Caldero--Chapoton function.

Caldero and Keller \cite{caldero2006triangulated} have extended the above correspondence to cluster algebras associated to acyclic quivers, that is, non-initial cluster variables of $\mathcal A(Q)$ are in bijection with real Schur roots in the root system associated to $Q$, and are again equal to the Caldero--Chapoton functions of the corresponding indecomposable rigid representations.

Gei\ss, Leclerc and Schr\"oer \cite{geiss2017quivers} have defined a class of Iwanaga--Gorenstein algebras $H$ associated to acyclic skew-symmetrizable matrices, generalizing the path algebras of acyclic quivers. These algebras are defined over arbitrary fields so certain geometric constructions valid for quivers carry over to them. The authors introduced \emph{locally free Caldero--Chapoton functions} for locally free $H$-modules and showed that in Dynkin cases those of locally free indecomposable rigid modules are exactly non-initial cluster variables \cite{geiss2018quivers}. Their proof however does not explicitly interpret mutations of cluster variables in terms of representations but actually relies on \cite{geiss2016quiversIII} a realization of the positive part of the enveloping algebra of a simple Lie algebra using locally free $H$-modules and a known connection between cluster algebras of Dynkin types and (dual) enveloping algebras \cite{yang2008cluster}.

In this paper, we study the recursion of locally free Caldero--Chapoton functions of modules under reflection functors. These functors, introduced in \cite{geiss2017quivers} for $H$-modules, generalize the classical Bernstein--Gelfand--Ponomarev reflection functors \cite{bernstein1973coxeter} for representations of Dynkin quivers. We show that this recursion coincides with cluster mutations that happen at a sink or source, leading to our main results:
\begin{enumerate}
    \item Non-initial cluster variables of a rank 2 cluster algebra are exactly locally free Caldero--Chapoton functions of locally free indecomposable rigid $H$-modules.
    \item In Dynkin cases, we obtain a new proof of the aforementioned correspondence in \cite{geiss2018quivers} which does not rely on results in \cite{geiss2016quiversIII} and \cite{yang2008cluster}.
    \item In general, any non-initial cluster variable obtained from almost sink and source mutations is expressed as the locally free Caldero--Chapoton function of a unique locally free indecomposable rigid $H$-module.
\end{enumerate}

We next provide a more detailed summary of this paper.

\subsection{Rank 2 cluster algebras}\label{subsection: rank 2}
Let $b$ and $c$ be two non-negative integers. The \emph{cluster algebra} $\mathcal A(b,c)$ is defined to be the subalgebra of $\mathbb Q(x_1, x_2)$ generated by \emph{cluster variables} $\{x_n\mid n\in \mathbb Z\}$ satisfying relations
\[
x_{n-1}x_{n+1} = \begin{cases} 1 + x_n^b \quad & \text{$n$ is odd}\\
1 + x_n^c \quad &\text{$n$ is even}.
\end{cases}
\]
Every cluster variable $x_n$ is viewed as a rational function of $x_1$ and $x_2$. The cluster algebras $\mathcal A(b,c)$ are said to be of rank 2 because the cardinality of each \emph{cluster} $\{x_n, x_{n+1}\}$ is 2.

Let $c_1$ and $c_2$ be two positive integers such that $c_1b = c_2c$. Let $g \coloneqq \gcd(b,c)$. Let $Q$ be the quiver
\[
\begin{tikzcd}
1\arrow[loop left, "\varepsilon_1"]\arrow[rr,out=45, in=135, shift left=.4ex, "\alpha_1"] \arrow[rr, out=25, in=155, "\alpha_2"]\arrow[rr,out=340, in=200,swap,"\alpha_g"] &\vdots & 2\arrow[loop right, "\varepsilon_2"]
\end{tikzcd}.
\]
Following \cite{geiss2017quivers}, we define $H = H(b,c,c_1,c_2)$ to be the path algebra $\mathbb CQ$ modulo the ideal
\[
I \coloneqq \langle \varepsilon_1^{c_1},\ \varepsilon_2^{c_2},\ \varepsilon_2^{b/g}\alpha_k - \alpha_k \varepsilon_1^{c/g}\mid k=1,2,\dots, g \rangle.
\]

Denote by $\rep H$ the category of finitely generated left $H$-modules. For any $M\in \rep H$ and $i\in \{1,2\}$, the subspace $M_i\coloneqq e_iM$ is a finitely generated module over the algebra $H_i\coloneqq e_iHe_i \cong \mathbb C[\varepsilon]/(\varepsilon^{c_i})$. We say that $M\in \rep H$ is \emph{locally free} (l.f. for short) if $M_i$ is a free $H_i$-module for $i=1,2$. For such $M$, we define its \emph{rank vector}
\[
\rk M \coloneqq (m_1, m_2)\in \mathbb N^2
\]
where $m_i$ denotes the rank of $M_i$ as a finitely generated free $H_i$-module. Let $E_1$ (resp. $E_2$) be the locally free module with rank vector $(1,0)$ (resp. $(0,1)$).

To any locally free $M\in \rep H$ with $\rk M = (m_1, m_2)$, we associated a Laurent polynomial
\begin{equation}\label{eq: cc function rank 2}
X_M(x_1, x_2) = x_1^{-m_1}x_2^{-m_2} \sum_{\mathbf r = (r_1, r_2)\in \mathbb N^2} \chi(\grlf(\mathbf r, M)) x_1^{b(m_2-r_2)} x_2^{cr_1} \in \mathbb Z[x_1^\pm, x_2^\pm],
\end{equation}
where $\grlf(\mathrm r, M)$ is the \emph{locally free quiver Grassmannian} (see \Cref{def: locally free grassmannian}) which is a quasi-projective complex variety parametrizing locally free submodules of $M$ with rank vector $\mathbf r$, and $\chi(\cdot)$ denotes the Euler characteristic in complex analytic topology. The Laurent polynomial $X_M$ is the \emph{locally free Caldero--Chapoton function} associated to $M$.

Our first main result is
\begin{theorem}[{\Cref{thm: main theorem rank 2}}]\label{thm: intro main theorem rank 2}
For $bc\geq 4$, there is a class of locally free indecomposable rigid $H$-modules $M(n)$ parametrized by $\{n\in \mathbb Z\mid \text{$n\leq 0$ or $n\geq 3$} \}$ such that 
\[
    X_{M(n)}(x_1, x_2) = x_n.
\]
In fact, this equality gives a bijection between all locally free indecomposable rigid $H$-modules (up to isomorphism) and non-initial cluster variables of $\mathcal A(b,c)$.
\end{theorem}

\begin{remark}
When $bc<4$, the cluster variables $x_n$ are periodic, that is, there are only finitely many distinguished $x_n$. These cases actually fall into another class of cluster algebras of \emph{finite types} (or \emph{Dynkin types}), which will be discussed in \Cref{subsection: lf cc function}.
\end{remark}

\begin{example}\label{ex: b=2 c=3}
Let $(b, c) = (2, 3)$ and $(c_1, c_2) = (3, 2)$. Then the algebra $H = H(b,c,c_1,c_2)$ is the path 
algebra of the quiver $Q$
\[
\begin{tikzcd}
1 \ar[loop left, "\varepsilon_1"] \ar[r, "\alpha"] &2 \arrow[loop right, "\varepsilon_2"]
\end{tikzcd}    
\]
modulo the relations $\varepsilon_1^3 = 0$ and $\varepsilon_2^2 = 0$. According to the construction in \Cref{rmk: construction of M(n)}, we list first a few $M(n)\in \replf H$ for $n\geq 3$.
\begin{enumerate}
    \item $M(3) = E_1$. One calculates the only non-empty quiver Grassmannians' Euler characteristics $\chi(\grlf((0,0), E_1)) = \chi(\grlf((1,0), E_1)) = 1$. Thus $X_{M(3)} = x_1^{-1}(1 + x_2^3) = x_3$.
    \item $M(4) = I_2$, the injective hull of $E_2$. It is obtained in a similar way as the module $N$ in \Cref{ex: reflection b=2 c=3}. It is easy to see that $\chi(\grlf((0,0), I_2)) = \chi(\grlf((2,1), I_2)) = \chi(\grlf((0,1), I_2)) = 1$ as in each case, the quiver Grassmannian consists of a single subrepresentation. We have $\chi(\grlf((1,1), I_2)) = \chi(\mathbb P^1) = 2$ according to \Cref{cor: euler char of fiber source}. Thus 
    \[
        X_{M(4)} = x_1^{-2}x_2^{-1}(x_1^2 + x_2^6 + 1 + 2x_2^3) = x_4.
    \]
    \item $M(5)$ is calculated in \Cref{ex: reflection b=2 c=3}. Computing $X_{M(5)}$ is not so easy, but via (the proof of) \Cref{lemma: recursion of f polynomial general}, we have
\[
X_{M(5)} = x_1^{-5}x_2^{-3} \left( x_1^6 + 3x_1^4(1+ x_2^3) + 3x_1^2(1 + x_2^3)^3 + (1 + x_2^3)^5 \right) = x_5.
\]
\end{enumerate}
\end{example}

\subsection{Higher rank cluster algebras}
Extending the construction of $\mathcal A(b,c)$ to any $n\times n$ integral skew-symmetrizable matrix $B$, there is an associated (coefficient-free) cluster algebra $\mathcal A(B) \subset \mathbb Q(x_1, \dots, x_n)$ with the initial seed \[\Sigma = (B, (x_1, \dots, x_n)).\] Here we briefly review the construction by Fomin and Zelevinsky \cite{fomin2002cluster}. The previous definition of $\mathcal A(b,c)$ in rank 2 corresponds to $B = \begin{bsmallmatrix} 0 & -b\\
c & 0\end{bsmallmatrix}$.

Let $\mathbb T_n$ be the infinite simple $n$-regular tree emanating from a given root $t_0$ such that the $n$ edges incident to any vertex are numbered by $\{1,\dots, n\}$. We associate $\Sigma$ to $t_0$, and inductively if $\Sigma_t  = (B^t = (b_{ij}^t), (x_{1;t}, \dots, x_{n;t}))$ is associated to some vertex $t\in \mathbb T_n$, then
\begin{equation}\label{eq: cluster mutation}
\Sigma_{t'}\coloneqq \mu_k(\Sigma_t) \coloneqq \left(\mu_k(B^t), (x_{1;t'}, \dots, x_{n;t'})\right)
\end{equation}
is associated to $t'$ for $t \frac{k}{\quad\quad} t'$ in $\mathbb T_n$, where $\mu_k(B^t)$ is Fomin--Zelevinsky's matrix mutation of $B^t$ in direction $k$ and
\[
\text{$x_{i;t'} \coloneqq x_{i;t}$ for $i\neq k$\quad and\quad $x_{k;t} \coloneqq x_{k;t}^{-1}\bigg(\prod_{i=1}^n x_{i;t}^{[b_{ik}^t]_+} + \prod_{i=1}^n x_{i;t}^{[-b_{ik}^t]_+}\bigg)$}.
\]
In this way, each $t\in \mathbb T_n$ is associated with a well-defined seed $(B^t, (x_{1;t},\dots, x_{n;t}))$ where $B^t$ is an $n\times n$ integral skew-symmetrizable matrix and each rational function $x_{i;t}\in \mathbb Q(x_1,\dots, x_n)$ is called a \emph{cluster variable}. The cluster algebra $\mathcal A(B)$ is then defined to be the subalgebra of $\mathbb Q(x_1,\dots, x_n)$ generated by all cluster variables. The exchange between $\Sigma_t$ and $\Sigma_{t'}$ for $t\frac{k}{\quad\quad} t'$ is usually called a \emph{cluster mutation}.

\subsection{Locally free Caldero--Chapoton formulas}\label{subsection: lf cc function}

Let $(C, D, \Omega)$ be an $n\times n$ symmetrizable Cartan matrix $C$, a symmetrizer $D$, and an acyclic orientation $\Omega$ (see \Cref{section: GLS algebras} for precise definitions). Gei\ss, Leclerc and Schr\"oer \cite{geiss2017quivers} have associated a finite dimensional $K$-algebra $H = H_K(C, D, \Omega)$ to the triple (where $K$ is a field), generalizing the path algebra of an acyclic quiver. Similar to the rank 2 case, there are locally free $H$-modules, forming the subcategory $\replf H \subset \rep H$. Analogously, each $M\in \replf H$ has its rank vector $\rk M\in \mathbb N^n$. Let $E_i$ be the locally free module with rank vector $\alpha_i \coloneqq (\delta_{ij} \mid j = 1, \dots n)$.

We define the bilinear form $\langle -,- \rangle_H \colon \mathbb Z^n\otimes \mathbb Z^n \rightarrow \mathbb Z$ such that on the standard basis $(\alpha_i)_{i=1}^n$,
\[
\langle \alpha_i, \alpha_j \rangle_H = \begin{cases}
c_i\quad &\text{if $i = j$},\\
c_ic_{ij}\quad &\text{if $(j,i)\in \Omega$},\\
0\quad &\text{otherwise}.
\end{cases}
\]
The skew-symmetrization of $\langle -,- \rangle_H$ (on the basis $(\alpha_i)_i$) defines a skew-symmetric matrix $DB$ (thus defining a skew-symmetrizable matrix $B = B(C, \Omega) = (b_{ij})$ actually having integer entries), i.e. explicitly, we have
\[
b_{ij} = \begin{cases}
c_{ij} &\text{if $(j,i)\in \Omega$},\\
-c_{ij} &\text{if $(i,j)\in \Omega$},\\
0 &\text{otherwise}.
\end{cases}
\]

\begin{definition}[{\cite[Definition 1.1]{geiss2018quivers}}]\label{def: cc function general}
For a locally free $H_\mathbb C(C, D, \Omega)$-module $M$, the associated \emph{locally free Caldero--Chapoton function} is
\[
X_M \coloneqq \sum_{\mathbf r\in \mathbb N^n} \chi(\grlf(\mathbf r, M))\prod_{i=1}^n v_i^{-\langle \mathbf r, \alpha_i \rangle_H - \langle \alpha_i, \rk M - \mathbf r\rangle_H}\in \mathbb Z[x_1^\pm, \dots, x_n^\pm],
\]
where $v_i \coloneqq x_i^{1/c_i}$.
\end{definition}

Suppose that $k\in \{1, \dots, n\}$ is a sink of $\Omega$ and let $s_k(H)\coloneqq H(C, D, s_k(\Omega))$ be the reflection of $H$ at $k$. There is the (sink) reflection functor (see \Cref{section: reflection functors})
\[
F_k^+\colon \rep H \rightarrow \rep s_k(H),
\]
which generalizes the classical BGP reflection functor on quiver representations.

The following proposition gives an algebraic identity on Caldero--Chapoton functions under reflections. It is the key recursion that makes connection with cluster mutations.

\begin{proposition}[{\Cref{lemma: recursion of f polynomial general}} and {\Cref{cor: recursion of cc function general}}]\label{prop: introduction cc function under reflection general}
Let $M$ be a locally free $H$-module such that the map $M_{k,\mathrm{in}}$ is surjective. Then the reflection $M'\coloneqq  F_k^+(M)\in \rep s_k(H)$ is also locally free, and
\[
X_M(x_1,\dots, x_n) = X_{M'}(x_1', \dots, x_n'),
\]
where 
\[
\text{$x_i'=x_i$ for $i\neq k$\quad and \quad $x_k' = x_k^{-1}\left( \prod_{i=1}^n x_i^{[b_{ik}]_+} + \prod_{i=1}^n x_i^{[-b_{ik}]_+}\right)$}.
\]
\end{proposition}

For $B = B(C, \Omega)$, an easy calculation shows that when $k$ is a sink or source, $\mu_k(B) = B(C, s_k(\Omega))$. This hints that the recursion in \Cref{prop: introduction cc function under reflection general} is closely related to cluster mutations as in (\ref{eq: cluster mutation}) at sink or source, which actually leads to our next main result

\begin{theorem}[{\Cref{thm: main theorem general}}]
For any non-initial cluster variable $x$ in $\mathcal A(B)$ obtained by almost sink and source mutations, there is a unique locally free indecomposable rigid $H$-module $M$ such that $x = X_M$.
\end{theorem}

If $C$ is of Dynkin type, it is known that any non-initial cluster variable can be obtained by almost sink and source mutations. Therefore our method provides a new proof of the following theorem of Gei\ss--Leclerc--Schr\"oer \cite[Theorem 1.2 (c) and (d)]{geiss2018quivers}.

\begin{theorem}[\Cref{thm: main theorem GLS Dynkin}]\label{thm: intro main theorem dynkin}
If $C$ is of Dynkin type, then the map $M\mapsto X_M$ induces a bijection between isomorphism classes of locally free indecomposable rigid $H$-modules and the non-initial cluster variables of $\mathcal A(B)$.
\end{theorem}

\subsection{Other related work} 

Caldero and Zelevinsky \cite{caldero2006laurent} studied how the Caldero--Chapoton functions of representations of generalized Kronecker quivers behave under reflection functors and used them to express cluster variables of skew-symmetric rank 2 cluster algebras. Our result in rank 2 can thus be seen as a generalization to the skew-symmetrizable case. 

We remark that the recursion in \Cref{prop: introduction cc function under reflection general} has already been achieved in the skew-symmetric case for any reflection, not necessarily at sink or source, of any quiver by Derksen--Weyman--Zelevinsky \cite{derksen2008quivers,derksen2010quivers}. 
Extending their theory, especially obtaining Caldero--Chapoton type formulas, to the skew-symmetrizable case in full generality remains an open problem; see for example \cite{demonet2010mutations, labardini2016strongly, Geuenich-LF1, Geuenich-LF2, lopez2019species, bautista2021potentials}.

There are several earlier work generalizing Caldero--Chapoton type formulas (or in the name of \emph{cluster characters}) to the skew-symmetrizable case. Demonet \cite{demonet2011categorification} has obtained cluster characters for acyclic skew-symmetrizable cluster algebras by extending \cite{geiss2011kac} to an equivariant version. Rupel \cite{rupel2011quantum, rupel2015quantum} has used representations of valued quivers over finite fields to obtain a quantum analogue of Caldero--Chapoton formula for quantum acyclic symmetrizable cluster algebras. The representation theories used in those work are however different from the one initiated in \cite{geiss2017quivers} which we follow in this paper. 

Fu, Geng and Liu \cite{fu2020cluster} have obtained locally free Caldero--Chapoton formulas for type $C_n$ cluster algebras with respect to not necessarily acyclic clusters. In the upcoming work \cite{LMpart2} with Labardini-Fragoso, we prove locally free Caldero--Chapoton formulas with respect to any cluster for cluster algebras associated to surfaces with boundary marks and orbifold points.

\subsection{Organization}
The paper is organized as follows. In \Cref{section: GLS algebras}, we recall the algebras $H(C, D, \Omega)$ defined by Gei\ss--Leclerc--Schr\"oer and review some necessary notions including locally free $H$-modules. In \Cref{section: reflection functors}, we review the definition of reflection functors for $H$-modules and their properties. In \Cref{section: lf cc functions} we study the reflections of $F$-polynomials of locally free modules, leading to a cluster type recursion of locally free Caldero--Chapoton functions. In \Cref{section: rank 2}, \Cref{section: dynkin case} and \Cref{section: beyond}, we apply the results obtained in \Cref{section: lf cc functions} to rank 2, Dynkin, and general cases respectively to obtain locally free Caldero--Chapoton formulas of cluster variables for skew-symmetrizable cluster algebras.

\section*{Acknowledgement}
The author would like to thank Daniel Labardini-Fragoso for helpful discussions. LM is supported by the Royal Society through the Newton International Fellowship NIF\textbackslash R1\textbackslash 201849.

\section{The algebras \texorpdfstring{$H(C, D, \Omega)$}{H(C,D,Omega)}}\label{section: GLS algebras}

In this section, we review the algebras $H(C, D, \Omega)$ defined in \cite{geiss2017quivers} and some relevant notions. 

\subsection{}
Let $(C, D, \Omega)$ be a \emph{symmetrizable Cartan matrix} $C$, a symmetrizer $D$ of $C$, and an acyclic orientation $\Omega$ of $C$. Let $I = \{1, \dots, n\}$. Precisely, the matrix $C = (c_{ij})\in \mathbb Z^{I\times I}$ satisfies that
\begin{itemize}
    \item $c_{ii} = 2$ for any $i\in I$ and $c_{ij}\leq 0$ for $i\neq j$, and
    \item there is some \emph{symmetrizer} $D = \mathrm{diag}(c_1, \dots, c_n)$ where $c_i\in \mathbb Z_{>0}$ for $i\in I$ such that $DC$ is symmetric.
\end{itemize}
An \emph{orientation} of $C$ is a subset $\Omega\subset I\times I$ such that 
\begin{itemize}
    \item $\{(i,j),(j,i)\}\cap \Omega \neq \varnothing$ if and only if $c_{ij}<0$, and
    \item if $(i,j)\in \Omega$, then $(j,i)\neq \Omega$.
\end{itemize}

Define $Q^\circ = Q^\circ (C, \Omega)\coloneqq (Q^\circ_0, Q^\circ_1, s, t)$ to be the quiver with
\begin{itemize}
    \item the set of vertices $Q_0^\circ \coloneqq I$, and
    \item the set of arrows $Q_1^\circ \coloneqq \{\alpha_{ij}^{(k)}\colon j\rightarrow i \mid (i,j)\in \Omega,\ k = 1,\dots, g_{ij}\}$ where $g_{ij}\coloneqq \gcd(c_{ij}, c_{ji})$.
\end{itemize}
We say $\Omega$ \emph{acyclic} if the quiver $Q^\circ$ is acyclic, i.e., has no oriented cycles. Define $Q = Q(C, \Omega)$ to be the quiver obtained from $Q^\circ$ by adding one loop $\varepsilon_i\colon i\rightarrow i$ to each vertex $i\in I$.

Following \cite{geiss2017quivers}, we define (over some ground field $K$) the algebra $H\coloneqq H_K(C, D, \Omega)$ to be the path algebra $KQ$ modulo the ideal generated by the relations
\begin{itemize}
    \item $\varepsilon_i^{c_i} = 0$ for $i\in I$,
    \item $\varepsilon_i^{f_{ji}} \alpha_{ij}^{(k)} = \alpha_{ij}^{(k)} \varepsilon_j^{f_{ij}}$ for $(i,j)\in \Omega$ and $k=1,\dots, g_{ij}$ where $f_{ij} \coloneqq  -c_{ij}/g_{ij}$.
\end{itemize}

The \emph{opposite orientation} of $\Omega$ is
\[
\Omega^* \coloneqq \{(i,j)\mid (j,i)\in \Omega\},
\]
which clearly is an orientation of $C$. We denote $H^* \coloneqq H(C, D, \Omega^*)$.

\subsection{} From now on, we will always assume that $\Omega$ is acyclic. For $i\in I$, let
\[
H_i \coloneqq e_iHe_i \cong K[\varepsilon_i]/(\varepsilon_i^{c_i}),
\]
which is a subalgebra of $H$. For $(i,j)\in \Omega$, define the $H_i$-$H_j$-sub-bimodule
\[
_iH_j \coloneqq \langle \alpha_{ij}^{(k)} \mid k = 1, \dots, g_{ij} \rangle \subset H.
\]
If $(j,i)\notin \Omega$, then $(i, j)\in \Omega^*$. Consider the algebra $H^* = H(C, D, \Omega^*)$. We define $_jH_i \coloneqq {_j(H^*)_i}$, which is an $H_j$-$H_i$-bimodule.

As a right $H_j$-module, ${_iH_j}$ is free of rank $-c_{ji}$ with the basis given by
\[
\{ \varepsilon_i^{f_{ji}-1-f}\alpha_{ij}^{(k)}\mid 0\leq f\leq f_{ji}-1,\ 1\leq k\leq g_{ij}\}.
\]
While as a left $H_i$-module, ${_iH_j}$ is free of rank $-c_{ij}$ with the basis
\[
\{ \alpha_{ij}^{(k)}\varepsilon_{j}^f \mid 0\leq k\leq g_{ij},\ 0\leq f\leq f_{ij}\}.
\]
There is then an $H_i$-$H_j$-bimodule isomorphism
\begin{equation}\label{eq: isomorphism rho}
\rho\colon {_iH_j} \rightarrow \Hom_{H_j}({_jH_i}, H_j),\quad \varepsilon_i^{f_{ji}-1-f}\alpha_{ij}^{(k)}\mapsto (\alpha_{ji}^{(k)} \varepsilon_i^f)^*    
\end{equation}
for $0\leq f\leq f_{ji}-1$ and $1\leq k\leq g_{ij}$.
For more details, we refer to \cite[Section 5.1]{geiss2017quivers}.

\subsection{}
Let $\rep H$ denote the category of finitely generated left $H$-modules. We will often treat $\rep H$ as the equivalent category of quiver representations of $Q$ satisfying relations in $I$. For $M\in \rep H$ and $i\in I$, the subspace $M_i\coloneqq e_iM$ is a finitely generated module over $H_i$.

\begin{definition}
We say that $M\in \rep H$ is \emph{locally free} if for each $i\in I$, the $H_i$-module $M_i$ is free, i.e. is isomorphic to $H_i^{\oplus r_i}$ for some $r_i\in \mathbb N$.
\end{definition}

Denote the full subcategory of locally free $H$-modules by $\replf H$. For $M\in \replf H$, define its \emph{rank vector}
\[
\rk M \coloneqq (r_1, \dots, r_n)\in \mathbb Z^I
\]
where $r_i$ stands for the rank of $e_iM$ as a free $H_i$-module. Let $E_i$ be the locally free $H$-module such that $\rk E_i = \alpha_i\coloneqq (\delta_{ij}\mid j\in I)\in \mathbb Z^I$.

We remark that $H_i$ itself is the only (non-zero) indecomposable projective (also injective) $H_i$-module. Any indecomposable $H_i$-module is isomorphic to the submodule $H_i \varepsilon_i^{k} \subset H_i$ for some $k \in \{ 0, \dots, c_i \}$.

\subsection{} Any $M\in \rep H$ is determined by the $H_i$-modules $M_i$ for $i\in I$ and the $H_i$-module homomorphisms
\[
M_{ij}\colon _iH_j\otimes_{H_j} M_j \rightarrow M_i,\quad (\alpha_{ij}^{(k)}, m) \mapsto M_{\alpha_{ij}^{(k)}}(m)
\]
for any $(i,j)\in \Omega$. We will later describe an $H$-module $M$ by specifying the data $(M_i, M_{ij})$. 

When there is no ambiguity, the subscript $H_j$ under the tensor product will be omitted, hence the simplified notation $_iH_j\otimes M_j$.

\section{Reflection functors}\label{section: reflection functors}

The Bernstein--Gelfand--Ponomarev reflection functors \cite{bernstein1973coxeter} are firstly defined to relate representations of an acyclic quiver $Q$ with that of the reflection of $Q$ at a sink or source vertex. These functors have been generalized to act on representations of $H(C, D, \Omega)$ in \cite{geiss2017quivers}. In this section, we recall their definitions and review some useful properties.

For an orientation $\Omega$ of $C$ and $i\in I$, the \emph{reflection of $\Omega$ at $i$} is the following orientation of $C$
\[
s_i(\Omega) \coloneqq \{(r,s)\in \Omega\mid i\notin \{r,s\}\} \cup \{(s,r)\in \Omega^* \mid i\in \{r, s\}\}.
\]
We denote $s_i(H) \coloneqq H(C, D, s_i(\Omega))$. Denote
\[
\Omega(i,-) \coloneqq \{j\in I\mid (i,j)\in \Omega\}\quad \text{and}\quad \Omega(-,i) \coloneqq \{j\in I\mid (j,i)\in \Omega\}.
\]
A vertex $i\in I$ is called a sink (resp. source) of $\Omega$ if it is a sink (resp. source) of the quiver $Q^\circ$, i.e. $\Omega(-,i) = \varnothing$ (resp. $\Omega(i,-) = \varnothing$). The only cases we will need are reflections at a sink or source.

\subsection{Sink reflection} Let $k$ be a sink of $\Omega$. In this subsection we define the \emph{sink reflection functor}
\[
F_k^+\colon \rep H \rightarrow \rep s_k(H).
\]
Denote $M_{(k,-)} \coloneqq \bigoplus_{j\in \Omega(k,-)} {_k H_j}\otimes M_j$. Consider the $H_k$-module morphism
\[
M_{k,\mathrm{in}}\coloneqq (M_{kj})_j \colon M_{(k,-)} \longrightarrow M_k.
\]
Let $N_k$ be $\ker M_{k,\mathrm{in}}$ as an $H_k$-module. Denote the inclusion $N_k\subset M_{(k,-)}$ by
\[
(\beta_{jk})_j\colon N_k \rightarrow M_{(k,-)}, \quad \beta_{jk} \colon N_k \rightarrow \,_kH_j\otimes M_j.
\]
The isomorphism $\rho\colon _kH_j\rightarrow \Hom_{H_j}(_jH_k, H_j)$ (\ref{eq: isomorphism rho}) induces the isomorphism
\[
\rho \colon \,_kH_j \otimes M_ j \rightarrow \Hom_{H_j}(_jH_k, M_j)
\]
Then further by the tensor-hom adjunction, we have
\[
\Hom_{H_k}(N_k, \,_kH_j\otimes M_j) \cong \Hom_{H_j} (_jH_k\otimes N_k, M_j),
\]
under which $\beta_{jk}$ corresponds explicitly to the map
\begin{equation}\label{eq: define sink reflection}
N_{jk}\colon _jH_k\otimes N_k \rightarrow M_j,\quad (\alpha, n) \mapsto \langle \alpha, \rho(\beta_{jk}(n)) \rangle \quad \text{for $n\in N_k$ and $\alpha\in {_jH_k}$}.    
\end{equation}

Now we define $F_k^+(M) = (N_r, N_{rs})$ with $(r,s)\in s_k(\Omega)$, where
\[
N_r\coloneqq \begin{cases}
M_r,\quad &\text{if $r\neq k$},\\
N_k,\quad &\text{if $r = k$}
\end{cases}\quad \text{and}
\]
\[
N_{rs}\coloneqq \begin{cases}
M_{rs}\quad &\text{if $(r,s)\in\Omega$ and $r\neq k$,}\\
N_{rk}\quad &\text{if $(r,s)\in \Omega^*$ and $s = k$}.
\end{cases}
\]

For a morphism $f = (f_i)_{i\in I} \colon M\rightarrow M'$ in $\rep H$, the morphism 
\[
    F_k^+(f) = (f_i')_{i\in I} \colon F_k^+(M) \rightarrow F_k^+(M')
\]
is defined by setting $f'_i = f_i$ for $i\neq k$ and $f_k'$ to be naturally induced between kernels. Thus $F_k^+$ is functorial.

\subsection{Source reflection}
For $k$ a source of $\Omega$, we define the \emph{source reflection functor}
\[
F_k^-\colon \rep H \rightarrow \rep s_k(H).
\]
Denote $M_{(-,k)} \coloneqq \bigoplus_{j\in \Omega(-,k)} \,_k H_j\otimes M_j$. Consider the $H_k$-module morphism
\[
M_{k,\mathrm{out}}\coloneqq (\overline M_{jk})_j \colon M_k\rightarrow M_{(-,k)}
\]
where each component $\overline  M_{jk} \colon M_k \rightarrow \,_kH_j\otimes M_j$ for $(k,j)\in \Omega$ is defined through the structure morphism $M_{jk}$ as follows. In fact, by the tensor-hom adjunction, we have the canonical isomorphism
\[
\Hom_{H_j}({_jH_k}\otimes M_k, M_j) \cong \Hom_{H_k}(M_k, \Hom_{H_j}({_jH_k}, M_j)),
\]
where the later is further identified with $\Hom_{H_k}(M_k, {_kH_j}\otimes M_j)$ induced by the isomorphism
\[
\rho\colon {_kH_j}\rightarrow \Hom_{H_j}({_jH_k}, H_j).
\]
Hence there is some $\overline M_{jk}\in \Hom_{H_k}(M_k, {_kH_j}\otimes M_j)$ that $M_{jk}\in \Hom_{H_j}({_jH_k}\otimes M_k, M_j)$ corresponds to.

Let $N_k$ be $\coker M_{k,\mathrm{out}}$ as an $H_k$-module. Denote the quotient $M_{(-,k)}\rightarrow N_k$ by
\[
(N_{kj})_j\colon M_{(-,k)}\rightarrow N_k, \quad N_{kj} \colon \,_kH_j\otimes M_j \rightarrow N_k.
\]
We define $F_k^-(M) = (N_r, N_{rs})$ with $(r,s)\in s_k(\Omega)$, where
\[
N_r\coloneqq \begin{cases}
M_r,\quad &\text{if $r\neq k$},\\
N_k,\quad &\text{if $r = k$}
\end{cases}\quad \text{and}
\]
\[
N_{rs}\coloneqq \begin{cases}
M_{rs}\quad &\text{if $(r,s)\in\Omega$ and $s\neq k$,}\\
N_{ks}\quad &\text{if $(r,s)\in \Omega^*$ and $r = k$}.
\end{cases}
\]
Analogously to $F_k^+$, it is clear that $F_k^-$ is also functorial.

\subsection{Some properties of reflection functors}

For $i\in I$, let $S_i$ be the simple $H$-module supported at the vertex $i$. Note that $S_i$ is at the same time the socle and the top (or head) of $E_i$. The following lemma is straightforward.

\begin{lemma}\label{lemma: equivalent conditions}
For any $M\in \rep H$, we have
\[
\Hom_H(M, S_k) = 0 \Leftrightarrow \Hom_H(M, E_k) = 0 \Leftrightarrow\ \text{$M_{k,\mathrm{in}}$ is surjective, and}
\]
\[
\Hom_H(S_k, M) = 0 \Leftrightarrow \Hom_H(E_k, M) = 0 \Leftrightarrow\ \text{$M_{k,\mathrm{out}}$ is injective}.
\]
\end{lemma}

\begin{proposition}[{\cite[Proposition 9.1 and Corollary 9.2]{geiss2017quivers}}]\label{prop: reflections are adjoint}
The pair of reflection functors 
\[
    F_k^+\colon \rep H\rightarrow \rep s_k(H)\quad \text{and}\quad F_k^-\colon \rep s_k(H)\rightarrow \rep H
\]
are (left and right) adjoint (additive) functors. They define inverse equivalences on subcategories
\begin{align*}
\mathcal T^H_k &\coloneqq \{M\in \rep H \mid \Hom_H(M, S_k) = 0\} \subset \rep H\quad \text{and}\\
\mathcal S^{s_k(H)}_k &\coloneqq \{M\in \rep s_k(H) \mid \Hom_{s_k(H)}(S_k, M) = 0\} \subset \rep s_k(H).
\end{align*}
\end{proposition}

We now focus on the actions of reflection functors on locally free modules.

\begin{lemma}[{\cite[Lemma 3.6]{geiss2018icm}}]\label{lemma: locally free and rigid implies surjective}
Suppose that $k$ is a sink (resp. source) of $\Omega$ and $M$ a locally free rigid $H$-module, with no direct summand isomorphic to $E_k$. Then we have $\Hom_H(M, E_k) = 0$ (resp. $\Hom_H(E_k, M) = 0$). In particular, the map $M_{k,\mathrm{in}}$ (resp. $M_{k,\mathrm{out}}$) is surjective (resp. injective).
\end{lemma}

\begin{proof}
The case where $k$ is a sink is \cite[Lemma 3.6]{geiss2018icm}. The other case is simply a dual version.
\end{proof}

Let $L = \mathbb Z^n$. We think of rank vectors of locally free $H$-modules as living in $L$ via $\mathbb N^n\subset \mathbb Z^n$. For $i\in I$, define the \emph{reflection}
\[
s_k\colon L\rightarrow L,\quad s_k(\alpha_i)  \coloneqq \alpha_i - c_{ki} \alpha_k\quad \text{for any $i\in I$}.
\]

\begin{proposition}[{\cite[Proposition 9.6]{geiss2017quivers}} and {\cite[Lemma 3.5]{geiss2018icm}}]\label{prop: reflection preserve locally free and rigid}
If $k$ is a sink (resp. source) of $\Omega$ and $M$ is a locally free rigid $H$-module, then $F_k^+(M)$ (resp. $F_k^-(M)$)
is locally free and rigid. If furthermore $M_{k,\mathrm{in}}$ is surjective (resp. $M_{k,\mathrm{out}}$ is injective), then
\[
\rk F_k^\pm (M) = s_k(\rk M).
\]
\end{proposition}

\begin{remark}\label{rmk: reflection preserve indecomposable}
We remark that the functor $F_k^+$ (resp. $F_k^-$) preserves indecomposability if $M_{k,\mathrm{in}}$ is surjective (resp. $M_{k,\mathrm{out}}$ is injective) by \Cref{prop: reflections are adjoint} and \Cref{lemma: equivalent conditions}.
\end{remark}

\section{Locally free Caldero--Chapoton functions}\label{section: lf cc functions}

\subsection{Locally free Caldero--Chapoton functions}
Let $H = H_\mathbb C(C, D, \Omega)$.

\begin{definition}\label{def: locally free grassmannian}
For $M\in \replf H$ and a rank vector $\mathbf r = (r_i)_{i\in I}$, the \emph{locally free quiver Grassmannian} is
\[
\grlf(\mathbf r, M) \coloneqq \{N \mid \text{$N$ is a locally free submodule of $M$ and $\rk N = \mathbf r$}\}.
\]
\end{definition}

\begin{remark}
It is clear that the set $\grlf (\mathbf r, M)$ can be realized as a locally closed subvariety of the product of ordinary Grassmannians $\prod_{i\in I}\mathrm{Gr}(c_ir_i, M_i)$. We take its analytic topology and denote by $\chi(\cdot)$ the Euler characteristic. 
\end{remark}

\begin{definition}
For $M\in \replf H$, we define its \emph{locally free $F$-polynomial} as
\[
F_M(y_1, \dots, y_n) \coloneqq \sum_{\mathbf r\in \mathbb N^n} \chi(\grlf(\mathbf r, M))\prod_{i=1}^n y_i^{r_i}\in \mathbb Z[y_1, \dots, y_n].
\]
\end{definition}

Recall that we have defined in \Cref{subsection: lf cc function} the bilinear form $\langle -,- \rangle_H \colon \mathbb Z^n \times \mathbb Z^n\rightarrow \mathbb Z$ and the skew-symmetrizable matrix $B = (b_{ij})$ associated to $(C,\Omega)$.

\begin{definition}
For $M\in \replf H$ with $\rk M = (m_i)_{i\in I}$, the associated \emph{locally free Caldero--Chapoton function} is the Laurent polynomial
\[
X_M \coloneqq \sum_{\mathbf r\in \mathbb N^n} \chi(\grlf(\mathbf r, M))\prod_{i=1}^n v_i^{-\langle \mathbf r, \alpha_i \rangle_H - \langle \alpha_i, \rk M - \mathbf r\rangle_H}\in \mathbb Z[x_1^\pm, \dots, x_n^\pm],
\]
where $v_i \coloneqq x_i^{1/c_i}$.
\end{definition}

\begin{remark}\label{rmk: express cc function by f polynomial}
Using the $F$-polynomial $F_M$, the Caldero--Chapoton function $X_M$ can be rewritten as
\begin{align*}
X_M &= \prod_{i=1}^n x_i^{-m_i + \sum_{j=1}^n [-b_{ij}]_+m_j} \cdot \sum_{\mathbf r\in \mathbb N^n} \chi(\grlf(\mathbf r, M))\prod_{i=1}^n x_i^{\sum_{j=1}^n b_{ij}r_j}\\
&= \prod_{i=1}^n x_i^{-m_i + \sum_{j=1}^n [-b_{ij}]_+m_j} \cdot F_M(\hat y_1, \dots, \hat y_n),
\end{align*}
where $\hat y_j = \prod_{i=1}^n x_i^{b_{ij}}$. Another way to write $X_M$ is
\[
X_M = \prod_{i=1}^n x_i^{-m_i} \cdot \sum_{\mathbf r\in \mathbb N^n} \chi(\grlf(\mathbf r, M))\prod_{i=1}^n x_i^{\sum_{j=1}^n [-b_{ij}]_+m_j + b_{ij}r_j}.
\]
We note that every term in the summation is an actual monomial since $[-b_{ij}]_+m_j+b_{ij}r_j\geq 0$ because we need $r_j\leq m_j$ for the quiver Grassmannian to be non-empty. Moreover, for $\mathbf r = 0$ and $\mathbf r = \rk M$, $\chi(\grlf(\mathbf r, M)) = 1$ and the two corresponding monomials are coprime. Therefore $\prod_{i=1}^nx_i^{m_i}$ can be characterized as the minimal denominator when expressing $X_M = f/g$ as a quotient of a polynomial $f$ and a monomial $g$. We thus call $\rk M$ the \emph{$\mathbf d$-vector} of the Laurent polynomial $X_M$.
\end{remark}

\begin{example}\label{ex: cc function of E_k}
For $k\in I$ and $E_k\in \replf H$, the only non-empty locally free quiver Grassmannians are $\grlf(0, E_k) = \{0\}$ and $\grlf(\alpha_k, E_k) = \{E_k\}$. Thus we have
\[
X_{E_k} = \prod_{i=1}^n v_i^{-\langle \alpha_i, \alpha_k \rangle_H} + \prod_{i=1}^n v_i^{-\langle \alpha_k, \alpha_i\rangle_H} = x_k^{-1}\left( \prod_{i=1}^n x_i^{[b_{ik}]_+} + \prod_{i=1}^n x_i^{[-b_{ik}]_+}\right).
\]
\end{example}

\subsection{The key recursion}
The following is the key proposition on the recursion of $F$-polynomials under reflections.

\begin{proposition}\label{lemma: recursion of f polynomial general} 
Let $M\in \replf H$ be of rank $(m_i)_{i\in I}$ and $k$ be a sink of $H$. Suppose that the map
\[
M_{k,\mathrm{in}}\colon \bigoplus_{j\in \Omega(k,-)}{_kH_j}\otimes M_j\rightarrow M_k
\]
is surjective. Then the reflection $M' \coloneqq F_k^+(M)$ is locally free over $s_k(H)$ with the rank vector $(m_i')_{i\in I}$ such that $m_i'=m_i$ for $i\neq k$ and 
\[
    m_k' = -m_k + \sum_{j\in \Omega(k,-)} b_{kj} m_j.
\]
Their $F$-polynomials satisfy the equation
\[
   (1 + y_k)^{-m_k} F_M(y_1,\dots, y_n) = F_{M'}(y_1', \dots, y_n').
\]
where 
\[
    y_i' = y_iy_k^{b_{ki}}/(y_k+1)^{b_{ki}}\ \text{for $i\neq k$,\quad and} \quad  y_k' = y_k^{-1}.
\]
\end{proposition}

\begin{proof}
The first half of the statement is simply a recast of \Cref{prop: reflection preserve locally free and rigid} without the hypothesis and conclusion on the rigidity. Explicitly, we observe that $M_k'$ and $M_k$ naturally fit into the following exact sequence (of free $H_k$-modules)
\begin{equation}\label{eq: an short exact sequence}
0 \longrightarrow  M_k' \xlongrightarrow{M'_{k,\mathrm{out}}} \bigoplus_{j\in \Omega(k,-)} {_kH_j}\otimes M_j  \xlongrightarrow{M_{k,\mathrm{in}}} M_k \longrightarrow 0.    
\end{equation}
Since $M_j$ is free over $H_j$ of rank $m_j$, we have that for any $j\in \Omega(k,-)$, the bimodule $_kH_j\otimes M_j$ is isomorphic to $_kH_j ^{\oplus m_j}$, thus a free left $H_k$-module of rank $b_{kj}m_j$. Then the calculation on $m_k'$ follows from the exact sequence.

Next we prove the recursion on $F$-polynomials.

\textbf{Step I}. Let $\mathbf e = (e_i)_i\in \mathbb N^I$ be a rank vector. Decompose $\grlf(\mathbf e, M)$ into constructible subsets $\tilde Z_{\mathbf e; r}(M)$ as follows. Let $N\subset M$ be a locally free submodule. Denote
\[
N_{(k,-)}\coloneqq \bigoplus_{j\in \Omega(k,-)} {_kH_j}\otimes N_j.
\]
Then $M_{k, \mathrm{in}}(N_{(k,-)})$ is an $H_k$-submodule of $M_k$. For $r\in \mathbb N$, let $\tilde Z_{\mathbf e;r}(M)$ be the subset of $\grlf(\mathbf e, M)$ consisting of $N\subset M$ such that
\[
\operatorname{rank} E(M_{k,\mathrm{in}}(N_{(k,-)})) = r.
\]
where $E(\cdot)$ denotes the injective hull (of an $H_k$-module).
Then $\grlf(\mathbf e, M)$ is a disjoint union of (finitely many) $\tilde Z_{\mathbf e;r}(M)$ when $r$ runs over $\mathbb N$ and thus
\[
\chi(\grlf(\mathbf e, M)) = \sum_{r\in \mathbb N} \chi(\tilde Z_{\mathbf e;r}(M)).
\]

\textbf{Step II}. Meanwhile for $M'\in \replf s_k(H)$, a rank vector $\mathbf e$ and $s\in \mathbb N$, let $\tilde X_{\mathbf e;s}(M')$ be the constructible subset of $\grlf(\mathbf e, M')$ consisting of locally free submodules $N\subset M'$ such that 
\[
\operatorname{rank}{F((M_{k,\mathrm{out}}')^{-1}(N_{(-,k)}))} = s,
\]
where we denote
\[
N_{(-,k)} \coloneqq \bigoplus_{j\in s_k(\Omega)(-,k)} {_kH_j}\otimes N_j,
\]
and $F(\cdot)$ stands for the (isomorphism class of) maximal free submodule of an $H_k$-module. Decomposing $\grlf(\mathbf e, M')$ into subsets $\tilde X_{\mathbf e;s}(M')$ where $s$ runs over $\mathbb N$,
we have
\[
\chi(\grlf(\mathbf e, M')) = \sum_{s\in \mathbb N} \chi(\tilde X_{\mathbf e;s}(M')).
\]

\textbf{Step III.}
Let $\mathbf e'$ denote the rank vector (with $n-1$ entries) obtained from $\mathbf e$ by forgetting the $k$-th component. Define for $r\in \mathbb N$ the subset
\[
Z_{\mathbf e';r}(M) \subset \prod_{i\neq k} \grlf(e_i, M_i) \coloneqq \prod_{i\neq k} \{(N_i)_{i\neq k} \mid N_i\subset M_i,\ \text{$N_i$ is free of rank $e_i$}\}
\]
such that
\begin{enumerate}
    \item $(N_i)_{i\neq k}$ is closed under the actions of arrows in (the quiver of) $H$ that are not incident to $k$;
    \item The injective hull of $M_{k,\mathrm{in}}(N_{(k,-)})$ is of rank $r$.
\end{enumerate}
There is the natural forgetful map
\[
    \pi\colon \tilde Z_{\mathbf e; r}(M) \rightarrow Z_{\mathbf e'; r}(M),\quad N \mapsto (N_i)_{i\neq k}.
\]
The fiber over a point $(N_i)_{i\neq k}\in Z_{\mathbf e';r}(M)$ is
\[
\mathcal V = \mathcal V (e_k, M_{k,\mathrm{in}}(N_{(k,-)}), M_k) \coloneqq \{N_k \mid \text{$M_{k,\mathrm {in}}(N_{(k,-)}) \subset N_k \subset M_k$, $N_k$ is free of rank $e_k$}\}.
\]
The space $\mathcal V$ is computed in \Cref{prop: euler char of fiber}. According to its description in the proof, $\mathcal V$ is determined (up to isomorphism) by the isomorphism class of $M_{k, \mathrm{in}}(N_{(k,-)})$ as $H_k$-module. This means we can if necessary decompose $Z_{\mathbf e'; r}(M)$ further into (finitely many) locally closed subsets $\coprod_{j\in J} Z_j$ so that each $\pi^{-1} (Z_j) \xrightarrow{\pi} Z_j$ is a fiber bundle. It is shown in \Cref{prop: euler char of fiber} that the Euler characteristic of a fiber $\mathcal V$ is $\binom{m_k-r}{e_k-r}$. Therefore we have
\begin{equation}\label{eq: euler char of V}
        \chi(\tilde Z_{\mathbf e; r}(M)) = \sum_{j\in J} \chi(\pi^{-1}(Z_j)) = \sum_{j\in J} \binom{m_k-r}{e_k-r} \chi(Z_j) = \binom{m_k-r}{e_k-r}\chi(Z_{\mathbf e'; r}(M)).
\end{equation}

For $M'$ and $s\in \mathbb N$, we define the subset 
\[
    X_{\mathbf e';s}(M') \subset \prod_{i\neq k} \grlf(e_i, M_i') \coloneqq \prod_{i\neq k} \{(N_i)_{i\neq k} \mid N_i\subset M'_i,\ \text{$N_i$ is free of rank $e_i$}\}
\]
such that
\begin{enumerate}
    \item $(N_i)_{i\neq k}$ is closed under the actions of arrows in (the quiver of) $s_k(H)$ not incident to $k$;
    \item The rank of a maximal free $H_k$-submodule of $(M'_{k,\mathrm{out}})^{-1}(N_{(-,k)})$ is $s$.
\end{enumerate}
Then we have the forgetful map
\[
\rho\colon \tilde X_{\mathbf e;s}(M') \rightarrow X_{\mathbf e';s}(M'),\quad N\mapsto (N_i)_{i\neq k},
\]
whose fiber at $(N_i)_{i\neq k}\in X_{\mathbf e';s}(M')$ is
\[
    \mathcal W = \mathcal W(e_k, (M'_{k,\mathrm{out}})^{-1}(N_{(-,k)})) \coloneqq \{ N_k \mid  N_k\subset (M'_{k,\mathrm{out}})^{-1}(N_{(-,k)})\subset M_k',\ \text{ $N_k$ is free of rank $e_k$} \},
\]
having Euler characteristic, according to \Cref{cor: euler char of fiber source}, $\chi(\mathcal W) = \binom{s}{e_k}$. Analogous to (\ref{eq: euler char of V}), we have
\begin{equation}\label{eq: euler char of W}
    \chi(\tilde X_{\mathbf e; s}(M')) = \binom{s}{e_k} \chi(X_{\mathbf e'; s}(M')).
\end{equation}

\textbf{Step IV}. Now recall that $M$ and $M'$ are reflections of each other, i.e. $M' = F_k^+(M)$, with rank vectors $(m_i)_i$ and $(m_i')_i$ respectively such that
\[
m_i = m_i'\text{ for $i\neq k$} \quad \text{and}\quad m_k + m_k' = \sum_{j\in \Omega(k,-)} b_{kj} m_j.
\]
We claim that for any $\mathbf e' = (e_i)_{i\neq k}$
\begin{equation}\label{eq: equality of two varieties}
X_{\mathbf e';s}(M') = Z_{\mathbf e'; r}(M)\quad \text{for}\quad r + s = \sum_{j\in \Omega(k,-)} b_{kj}e_j.    
\end{equation}
In fact, let $N_i$ be a free submodule (of rank $e_i$) of $M_i$ for any $i\neq k$ and then we have from (\ref{eq: an short exact sequence}) the following short exact sequence of $H_k$-modules
\[
0 \longrightarrow (M'_{k,\mathrm{out}})^{-1}(N_{(-,k)})\longrightarrow N_{(k,-)} \longrightarrow M_{k,\mathrm{in}}(N_{(k,-)}) \longrightarrow 0,
\]
where by our abuse of notation, the $H_k$-module $N_{(-,k)}$ (which is for the orientation $s_k(\Omega)$) is actually the same as $N_{(k,-)}$ for $\Omega$. 
Denote $A = (M'_{k,\mathrm{out}})^{-1}(N_{(-,k)})$, $B = N_{(k, -)}$, and $C = M_{k,\mathrm{in}}(N_{(k,-)})$. The injection $A\rightarrow B$ then factors through the injective hull $E(A)$ of $A$. Thus $C\cong E(A)/A \oplus B/E(A)$ where $B/E(A)$ is free. The number of indecomposable summands of $E(A)/A$ is easily seen to be $\operatorname{rank} E(A) - \operatorname{rank} F(A)$. The number of indecomposable summands of $C$, which equals $\operatorname{rank} E(C)$, is just $\operatorname{rank} B - \operatorname{rank} F(A)$. We now have the equality
\[
\operatorname{rank} E(M_{k,\mathrm{in}}(N_{(k,-)})) + \operatorname{rank} {F((M_{k, \mathrm{out}}')^{-1}(N_{(-,k)}))} = \operatorname{rank} N_{(k,-)} = \sum_{j\in \Omega(k,-)} b_{kj}e_j.
\]
Then one sees from their definitions that $X_{\mathbf e', s}(M')$ and $Z_{\mathbf e';r}(M)$ for any $r$ and $s$ such that $r + s  = \operatorname{ rank } N_{(k,-)}$ define the exact same tuples $(N_i)_{i\neq k}$ in $\prod_{i\neq k} \grlf(e_i, M_i)$, hence (\ref{eq: equality of two varieties}).

Now we rewrite the $F$-polynomials as
\begin{align*}
    F_M(y_1, \dots, y_n) & = \sum_{\mathbf e\in \mathbb N^n} \sum_{r\in \mathbb N} \chi(\tilde Z_{\mathbf e;r}(M)) y^\mathbf e \\
    & \overset{(\ref{eq: euler char of V})}{=} \sum_{\mathbf e\in \mathbb N^n} \sum_{r\in \mathbb N} \binom{m_k-r}{e_k -r}\chi(Z_{\mathbf e';r}(M)) y^\mathbf e\\
    & = \sum_{\mathbf e'\in \mathbb N^{n-1}}\sum_{r\in \mathbb N}\chi(Z_{\mathbf e';r}(M))y_k^r(1+y_k)^{m_k-r}\prod_{i\neq k}y_i^{e_i},\\
    \\
    F_{M'}(z_1, \dots, z_n) & = \sum_{\mathbf e\in \mathbb N^n} \sum_{s\in \mathbb N} \chi(\tilde X_{\mathbf e;s}(M')) z^\mathbf e \\
    & \overset{(\ref{eq: euler char of W})}{=} \sum_{\mathbf e\in \mathbb N^n} \sum_{s\in \mathbb N} \binom{s}{e_k} \chi(X_{\mathbf e';s}(M')) z^\mathbf e\\
    & = \sum_{\mathbf e'\in \mathbb N^{n-1}}\sum_{s\in \mathbb N}\chi(X_{\mathbf e';s}(M'))(1+z_k)^{s}\prod_{i\neq k} z_i^{e_i}.
\end{align*}
Let $z_i = y_i' \coloneqq y_iy_k^{b_{ki}}/(y_k+1)^{b_{ki}}$ for $i\neq k$ and $z_k = y_k' \coloneqq y_k^{-1}$. Finally we have
\begin{align*}
F_{M'}(y_1', \dots, y_n') &= \sum_{\mathbf e'\in \mathbb N^{n-1}}\sum_{s\in \mathbb N} \chi(X_{\mathbf e';s}(M')) \left(\frac{y_k}{1+y_k}\right)^{-s+\sum_{j\neq k}b_{kj}e_j}\prod_{i\neq k}y_i^{e_i}\\
&\overset{(\ref{eq: equality of two varieties})}{=} \sum_{\mathbf e'\in \mathbb N^{n-1}}\sum_{r\in \mathbb N} \chi(Z_{\mathbf e';r}(M')) \left(\frac{y_k}{1+y_k}\right)^{r}\prod_{i\neq k}y_i^{e_i}\\
&=(1+y_k)^{-m_k}F_M(y_1,\dots, y_n).
\end{align*}
\end{proof}

\begin{corollary}\label{cor: recursion of cc function general}
In the setting of the above \Cref{lemma: recursion of f polynomial general},  we have
\[
X_M(x_1,\dots, x_n) = X_{M'}(x_1', \dots, x_n'),
\]
where $x_i'=x_i$ for $i\neq k$ and
\[
    x_k' = x_k^{-1}\left( \prod_{i\in I} x_i^{[b_{ik}]_+} + \prod_{i\in I} x_i^{[-b_{ik}]_+}\right) = x_k^{-1}\left(1 + \prod_{i\in \Omega(k,-)} x_i^{-b_{ik}} \right).
\]
\end{corollary}

\begin{proof}
We first derive from \Cref{lemma: recursion of f polynomial general} that
\[
(1 + \hat y_k)^{-m_k} F_M(\hat y_1, \dots, \hat y_n) = F_{M'}(\hat y_1', \dots, \hat y_n'),
\]
where $\hat y_j \coloneqq \prod_{i=1}^n x_i^{b_{ij}}$ and $\hat y_j' \coloneqq \prod_{i=1}^n (x_i')^{b'_{ij}}$ where $(b'_{ij}) \coloneqq  \mu_k(B)$ or explicitly $b'_{ij} = -b_{ij}$ if $i=k$ or $j=k$ and $b'_{ij} = b_{ij}$ otherwise. In fact, this equality directly follows from the algebraic equations
\[
y_j'\vert_{y_i \leftarrow \hat y_i, i=1,\dots, n} = \hat y_j',\quad j\in I.
\]

Then we spell out the two sides of the desired equation in the form of \Cref{rmk: express cc function by f polynomial}. Now it amounts to show
\[
    \prod_{i=1}^n x_i^{-m_i + \sum_{j=1}^n [-b_{ij}]_+m_j} = (1 + \hat y_k)^{-m_k} \prod_{i=1}^n (x_i')^{-m'_i + \sum_{j=1}^n [-b'_{ij}]_+m'_j},
\]
which is straightforward to check.
\end{proof}

We finish this section by proving the following proposition (and \Cref{cor: euler char of fiber source}) which has been used in the proof of \Cref{lemma: recursion of f polynomial general}.

\begin{proposition}\label{prop: euler char of fiber}
Let $M$ be a (finitely generated) $\mathbb C[\varepsilon]/(\varepsilon^n)$-module whose (any) maximal free submodule $F(M)$ is of rank $m$. Let $L\subset M$ be a submodule such that $E(L)$ the injective hull of $L$ is of rank $\ell$. Assume further that $L$ is contained in a free submodule of $M$. Then for any integer $e$ between $\ell$ and $m$, the variety
\[
\mathcal V = \mathcal V(e, L, M) \coloneqq \{N\mid L\subset N \subset M,\ \text{$N$ is free of rank $e$}\}
\]
has Euler characteristic $\binom{m-\ell}{e-\ell}$.
\end{proposition}

\begin{proof}

For any $\mathbb C[\varepsilon]/(\varepsilon^n)$-module $M$, let for $0\leq k \leq n$,
\[
    M^{(k)} \coloneqq \varepsilon^{n-k}(M)\quad \text{and} \quad M_{(k)} \coloneqq \ker (\varepsilon^k\colon M\rightarrow M).
\]
It is clear that $M^{(k)}$ is contained in $M_{(k)}$. These vector spaces fit into short exact sequences
\[
0\longrightarrow M_{(k)} \longrightarrow M \xlongrightarrow{\varepsilon^k} M^{(n-k)} \longrightarrow 0
\]
and filtrations
\[
    0 = M^{(0)}\subset M^{(1)} \subset \cdots \subset M^{(n)} = M,\quad 0 = M_{(0)}\subset M_{(1)} \subset \cdots \subset M_{(n)} = M.
\]
For example, if $\operatorname{rank} F(M) = m$, then $M^{(1)} \cong \mathbb C^m$; if $\operatorname{rank} E(M) = \ell$, then $M_{(1)} \cong \mathbb C^\ell$.

Consider for $1\leq k\leq n$, the variety $\mathcal F_k$ of flags in $M$
\[
0 = N^{(0)}\subset N^{(1)} \subset \cdots \subset N^{(k)} \subset M
\]
such that for $0\leq i\leq k$, each $N^{(i)}$ is a $\mathbb C$-vector subspace of dimension $ie$ satisfying the following conditions:
\begin{enumerate}
    \item $L_{(i)} \subset N^{(i)} \subset M^{(i)}$;
    \item $\varepsilon (N^{(i)}) = N^{(i-1)}$.
\end{enumerate}
A point in $\mathcal F_n$ clearly determines a submodule $N^{(n)}$ of $M$ containing $L$, which is free of rank $e$ simply for dimension reasons. A free basis can be obtained by choosing a (vector space) basis of $N^{(1)}$ and taking a lift in $N^{(n)}$ through $\varepsilon^{n-1}$. Sending $N\in \mathcal V$ to the filtration given by $(N^{(i)})_{i=1}^n$ induces an isomorphism from $\mathcal V$ to $\mathcal F_n$. There are maps 
$\pi_{k+1, k}\colon \mathcal F_{k+1} \rightarrow \mathcal F_{k}$ forgetting the largest subspace $N^{(k+1)}$ in a flag. We next show that
\begin{enumerate}
    \item $\mathcal F_1$ is isomorphic to the Grassmannian $\mathrm{Gr}(e-\ell, m-\ell)$;
    \item each $\pi_{k+1,k}$ is a fiber bundle with fiber being an affine space.
\end{enumerate}

The assumption that $E(L)$ is of rank $\ell$ implies $L_{(1)} \cong \mathbb C^{\ell}$, and that $F(M)$ is of rank $m$ implies $M^{(1)} \cong \mathbb C^m$. Then $\mathcal F_1$ is the space $\{N^{(1)}\mid L_{(1)}\subset N^{(1)} \subset M^{(1)},\ \dim N^{(1)} = e\}$, clearly isomorphic to $\mathrm{Gr}(e-\ell, m-\ell)$.

Let $(N^{(i)})_{i=1}^k$ be a point in $\mathcal F_k$. We collect several auxiliary vector subspaces of $M^{(k+1)}$) to be used later.
\begin{itemize}
    \item 
    Let $P^{(k)}$ be the preimage of $N^{(k)}$ in $M^{(k+1)}$ of the (surjective) map $\varepsilon \colon M^{(k+1)} \rightarrow M^{(k)}$. It fits in the short exact sequence
    \[
        0 \longrightarrow K^{(k+1)} \longrightarrow P^{(k)} \xlongrightarrow{\varepsilon} N^{(k)} \longrightarrow 0
    \]
    where $K^{(k+1)} \coloneqq M^{(k+1)} \cap M_{(1)}$ which is the kernel of $\varepsilon \colon M^{(k+1)} \rightarrow M^{(k)}$. Let $m_{k+1} \coloneqq \dim K^{(k+1)}$. We have $\dim P^{(k)} = ek + m_{k+1}.$ The subspace $P^{(k)}$ contains $N^{(k)}$ because $\varepsilon(N^{(k)}) = N^{(k-1)} \subset N^{(k)}$. Since $N^{(1)}\subset M^{(1)} = M^{(1)}\cap M_{(1)} \subset K^{(k+1)}$, we have $e\leq m \leq m_{k+1}$. In fact, the sequence $(m_{k+1})_{k=0}^{n-1}$ is increasing because of the filtration $(M^{(k+1)})_{k=0}^{n-1}$.
    
    \item 
    We have that $L_{(k+1)}$ is contained in $P^{(k)}$ since $\varepsilon(L_{(k+1)})$ is in $L_{(k)}\subset N^{(k)}$. Denote $\ell_{k+1} = \dim L_{(k+1)}/L_{(k)}$. As $\dim L_{(k)}$ is complementary to $\dim L^{(n-k)}$ in $L$, the sequence $(\ell_{k+1})_{k=0}^{n-1}$ is decreasing. Thus $\ell_{k+1}\leq \ell_1 = \ell\leq e$.
    
    \item 
    The two subspaces $L_{(k+1)}$ and $N^{(k)}$ intersect to give exactly $L_{(k)}$. Their span $W^{(k+1)}\coloneqq L_{(k+1)} + N^{(k)}$ thus has dimension $ke+\ell_{k+1}$. Notice $W^{(k+1)}\cap K^{(k+1)} = N^{(1)}$. In fact, for any $a\in N^{(k)}$ and $b\in L_{(k+1)}$ such that $\varepsilon(a+b) = 0$, we have $\varepsilon(a) = \varepsilon(-b)$ in $N^{(k-1)}\cap L_{(k)} = L_{(k-1)}$. Then $b$ belongs to $L_{(k)} \subset N^{(k)}$ and $a+b\in N^{(k)}$. Thus $a+b \in \ker(\varepsilon\colon N^{(k)} \rightarrow N^{(k-1)}) = N^{(1)}$.
\end{itemize}
Now the fiber of $\pi_{k+1, k}$ at $(N^{(i)})_{i=1}^k$ is
\[
    X_{k+1}\coloneqq \left\{N^{(k+1)}\mid W^{(k+1)}\subset N^{(k+1)} \subset P^{(k)},\ \dim N^{(k+1)} = (k+1)e,\ N^{(k+1)}\cap K^{(k+1)} = N^{(1)} \right\}
\]
We claim that $X_{k+1}$ is isomorphic to the affine space $\mathbb A^{(e-\ell_{k+1})\cdot (m_{k+1}-e)}$, which follows from
\begin{lemma}
Let $0\leq e\leq a\leq b\leq c$ and $e\leq d$ be non-negative integers such that $b + d - e = c$. Let $E\subset A\subset C$ be vector spaces of dimensions $e$, $a$, and $c$ respectively. Let $D\subset C$ be a subspace of dimension $d$ such that $D \cap A = E$. Then the space
\[
X \coloneqq \{B\mid A\subset B\subset C,\ \dim B = b,\ B\cap D = E\}
\]
is isomorphic to the affine space $\mathbb A^{(b-a)\cdot (d-e)}$.
\end{lemma}
\begin{proof}
Quotient by $E$ the common subspace for all, the space $X$ is identified with
\[
X'\coloneqq \{B' \mid A/E \subset B' \subset C/E,\ \dim B' = b-e,\ B'\cap D/E = 0 \}.
\]
Now $D/E \cap A/E = 0$ in $C/E$. Consider the quotient map $\pi \colon C/E \rightarrow (C/E)/(A/E)$, which induces an isomorphism from $X'$ to
\[
X''\coloneqq \{B'' \mid B''\subset (C/E)/(A/E),\ \dim B'' = b-a,\ B''\cap (D/E) = 0\},
\]
where $D/E$ denotes the image of $D/E$ under $\pi$. Now we have $\dim {(C/E)/(A/E)} = c-a$, $\dim {(D/E)} = d-e$, and thus
\[
\dim {(C/E)/(A/E)} = \dim {(D/E)} + \dim B''.
\]
It is then standard that the space $X''$ (of vector subspaces of a given dimension transversal to a fixed vector subspace of the complimentary dimension) is isomorphic to $\mathbb A^{(b-a)\cdot (d-e)}$.
\end{proof}

Now in the context of the above lemma, let 
\[
\text{$A = W^{(k+1)}$, $B = N^{(k+1)}$, $C = P^{(k)}$, $D = K^{(k+1)}$, $E = N^{(1)}$, and}
\]
\[
a = ke + \ell_{k+1},\quad b = (k+1)e,\quad c = ke + m_{k+1},\quad d = m_{k+1},\quad e=e.
\]
We see that $X_{k+1} \cong \mathbb A^{(e-\ell_{k+1})\cdot (m_{k+1}-e)}$. Now that the fiber of each $\pi_{k+1,k}$ is an affine space, the variety $\mathcal V \cong \mathcal F_n$ is then homotopic to the base $\mathcal F_1 \cong \mathrm{Gr}(e-\ell, m-\ell)$.
The desired result on $\chi(\mathcal V)$ follows because $\chi(\mathrm{Gr}(e-\ell, m-\ell)) = \binom{m-\ell}{e-\ell}$.

\end{proof}

\begin{corollary}\label{cor: euler char of fiber source}
Let $M$ be a $\mathbb C[\varepsilon]/\varepsilon^n$-module whose (any) maximal free submodule is of rank $m$. Then for any $0\leq e\leq m$, the variety
\[
\mathcal W(e, M) \coloneqq \{N\mid N\subset M,\ \text{$N$ is free of rank $e$} \}
\]
has Euler characteristic $\binom{m}{e}$.
\end{corollary}

\begin{proof}
In the setting of \Cref{prop: euler char of fiber}, letting $L = 0$, the result follows.
\end{proof}

\section{The rank 2 case}\label{section: rank 2}

The purpose of this section is to prove \Cref{thm: intro main theorem rank 2} (\Cref{thm: main theorem rank 2}). In fact, the construction of the algebra $H$ in \Cref{subsection: rank 2} can be seen as within the general framework introduced in \Cref{section: GLS algebras}, which we explain in below.

Let $I = \{1, 2\}$. Let $C$ be the Cartan matrix
\[
     \begin{pmatrix}
    2 & -b\\
    -c & 2
\end{pmatrix}\quad \text{for $b,c\in \mathbb Z_{\geq 0}$}.
\]
There are only two possible orientations, i.e. $\Omega = \{(1,2)\}$ or $\{(2,1)\}$. We take $\Omega = \{(2,1)\}$. The associated matrix $B$ is $\begin{psmallmatrix}
 0 & -b \\
 c & 0
\end{psmallmatrix}.$ Assume in the rest of this section that $bc\geq 4$. The rest cases are of Dynkin types which will be covered in the next \Cref{section: dynkin case}.

Let $D = \mathrm{diag}(c_1, c_2)$ be a symmetrizer of $C$. One easily sees that $H\coloneqq H(C, D,\Omega)$ is the same as $H(b,c,c_1,c_2)$ defined in \Cref{subsection: rank 2}, where we denote the arrow $\alpha_k$ there by $\alpha^{(k)}_{21}$. It is clear that
\[
s_1(\Omega) = s_2(\Omega) = \Omega^* = \{(1,2)\}.
\]
Denote $H^* \coloneqq H(C, D, \Omega^*)$. Then there are reflection functors
\[
F^\pm\colon \rep H \rightarrow \rep H^*,\quad F^\pm \colon \rep H^* \rightarrow \rep H.
\]
We omit the subscripts in the reflection functors since the sign $\pm$ already specifies which vertex the reflection is performed at. Next we define a class of modules obtained from iterative reflections.

\begin{definition}
We define for $n\geq 0$ the following $H$-modules
\[
M(n+3) \coloneqq \begin{cases} (F^+)^n E_1 \quad &\text{if $n$ is even,}\\
(F^+)^n E_2 \quad &\text{if $n$ is odd}
\end{cases}\quad\text{and}\quad
M(-n) \coloneqq \begin{cases} (F^-)^n E_2 \quad &\text{if $n$ is even,}\\
(F^-)^n E_1 \quad &\text{if $n$ is odd.}
\end{cases}
\]
\end{definition}

\begin{remark}\label{rmk: construction of M(n)}
Let us clarify the above construction of $M(n+3)$. For any $n\geq 0$ and $0\leq k\leq n$, let 
\[
H^{(k)} \coloneqq \begin{cases}
H &\text{if $k$ is even,}\\
H^* &\text{if $k$ is odd.}
\end{cases}
\]
Now we have a sequence of functors $F^{(k)}\coloneqq F^+\colon \rep H^{(k+1)}\rightarrow \rep H^{(k)}$ for $0\leq k\leq n-1$. Then $M(n+3)$ is obtained by iteratively applying $F^{(k)}$, i.e.
\[
M(n+3) \coloneqq \begin{cases} F^{(0)}\circ F^{(1)} \circ \cdots \circ F^{(n-1)} (E_1) \quad &\text{if $n$ is even,}\\
F^{(0)}\circ F^{(1)} \circ \cdots \circ F^{(n-1)} (E_2) \quad &\text{if $n$ is odd}.
\end{cases}
\]
The modules $M(-n)$ are defined using $F^-$ in a similar way.
\end{remark}

\begin{lemma}\label{lemma: incoming map is surjective rank 2}
For any $n\geq 0$, the $H$-module $M = M(n+3)$ (resp. $M(-n)$) is locally free, indecomposable and rigid. The map $M_{2,\mathrm{in}}$ (resp. $M_{1,\mathrm{out}}$) is surjective (resp. injective).
\end{lemma}

\begin{proof}
It follows from \Cref{prop: reflection preserve locally free and rigid} that any $M(n+3)$ or $M(-n)$ is locally free and rigid because so is $E_1$ or $E_2$.

Now assume that for any $0\leq k \leq n$ the modules $M(k+3)$ and $M(-k)$ are all indecomposable and that the map $M_{1,\mathrm{out}}$ is surjective for $M(k+3)$ and $M_{2,\mathrm{in}}$ is injective for any $M(-k)$. Denote the rank vectors by $\alpha(n) \coloneqq \rk M(n)$. Now by the construction of $M(n+4)$ and $M(-(n+1))$ and \Cref{prop: reflection preserve locally free and rigid}, we have that these two modules are locally free and rigid, and 
\[
\alpha(n+4) = s_1s_2s_1\cdots (\alpha_{\langle n + 4 \rangle}),\quad \alpha(-(n+1)) = s_2s_1s_2\cdots (\alpha_{\langle -(n+1)\rangle}),
\]
where $\langle n \rangle \in \{1, 2\}$ is congruent to $n$ modulo $2$. It is then known that (in the case $bc\geq 4$) both $\alpha(n+4)$ and $\alpha(-(n+1))$ are real positive roots of $C$ (other than the simple roots $\alpha_1$ and $\alpha_2$) and in particular are strictly positive linear combinations of $\alpha_1$ and $\alpha_2$; see for example \cite[Section 3.1]{sherman2004positivity}. By \Cref{rmk: reflection preserve indecomposable}, $M(n+4)$ and $M(-(n+1))$ are also indecomposable. So they cannot have any summand isomorphic to $E_1$ or $E_2$. Now by \Cref{lemma: locally free and rigid implies surjective}, the induction is completed.
\end{proof}

\begin{remark}\label{rmk: lf ind rigid modules rank 2}
In fact, by \cite{geiss2018rigid}, locally free indecomposable rigid $H$-modules are parametrized by their rank vectors as real Schur roots of $C$ (depending on $\Omega$). Since the rank vectors $\alpha(n) = \rk M(n)$ for $n\leq 0$ and $n\geq 3$ are exactly the real Schur roots (see for example \cite{sherman2004positivity}), we know that $\{M(n)\mid \text{$n\leq 0$ or $n\geq 3$}\}$ fully lists locally free indecomposable rigid $H$-modules.
\end{remark}

\begin{example}\label{ex: reflection b=2 c=3}
Recall the algebra $H$ considered in \Cref{ex: b=2 c=3} where $b=2$ and $c=3$. We calculate $N \coloneqq F_2^+(E_1)\in \rep s_2(H)$ as follows. First, we have
\[
N_1 = (E_1)_1 = H_1\quad \text{and} \quad N_2 = {_2H_1}\otimes H_1 = (H_2 \cdot \alpha_{21})\oplus (H_2 \cdot \alpha_{21} \varepsilon_1) \oplus (H_2\cdot \alpha_{21} \varepsilon_1^2).
\]
The structure map $N_{12}\colon {_1H_2} \otimes N_2\rightarrow N_1$ is given by
\[
\alpha_{12}\otimes \alpha_{21}h \mapsto 0\quad \text{and} \quad \alpha_{12}\otimes \varepsilon_2 \alpha_{21}h \mapsto h
\]
for any $h\in H_1$. Then one sees $N \cong I_1 \in \rep s_2(H)$, which is locally free, indecomposable and rigid, and $\rk N = (1,3) = s_2(\rk E_1)$.

We next calculate $M = M(5) \coloneqq F_1^+F_2^+(E_1) = F_1^+(N)$. By definition, $M_2 = N_2$ and
\[
M_1  =  \ker ({_1H_2\otimes N_2} \xrightarrow{N_{12}} N_1),
\]
which is a free $H_1$-module of rank 5 having the basis
\begin{gather*}
e_1\coloneqq \alpha_{12}\otimes \alpha_{21},\quad e_2\coloneqq \alpha_{12}\otimes \alpha_{21}\varepsilon_1,\quad e_3\coloneqq \alpha_{12} \otimes \alpha_{21}\varepsilon_1^2,\\
e_4\coloneqq \alpha_{12} \otimes \varepsilon_2\alpha_{21}\varepsilon_1 - \varepsilon_1\alpha_{12} \otimes \varepsilon_2\alpha_{21},\quad e_5 \coloneqq \alpha_{12} \otimes \varepsilon_2\alpha_{21}\varepsilon_1^2 - \varepsilon_1\alpha_{12}\otimes \varepsilon_2 \alpha_{21}\varepsilon_1.
\end{gather*}
Thus $M_1 \cong \mathbb C^{15}$ as a vector space with the basis $\{\varepsilon_1^ke_j\mid 1\leq j \leq 5,\ 0\leq k\leq 2\}$. The action of $\alpha_{21}$ on this basis is calculated in table below (only non-zero terms shown). 

\begin{center}
\begin{tabular}{ c|c|c|c|c|c|c|c } 
     $ \varepsilon_1^ke_j$ & $ \varepsilon_1^2e_1 $ & $ \varepsilon_1^2 e_2$ & $ \varepsilon_1^2 e_3$ & $ \varepsilon_1e_4$ & $ \varepsilon_1^2e_4$ & $\varepsilon_1e_5$ & $\varepsilon_1^2 e_5$\\ 
     \hline
    $M_{\alpha_{21}}(\cdot)$ & $\alpha_{21}$ & $\alpha_{21}\varepsilon_1$ & $\alpha_{21}\varepsilon_1^2$ & $-\varepsilon_2\alpha_{21}$ & $\varepsilon_2\alpha_{21}\varepsilon_1$ & $-\varepsilon_2\alpha_{21}\varepsilon_1$ & $\varepsilon_2\alpha_{21}\varepsilon_1^2$\\ 
\end{tabular}
\end{center}
For example, we have by (\ref{eq: define sink reflection}) that
\[
M_{\alpha_{21}}(\varepsilon_1e_5) = \langle \alpha_{21}, \rho(\varepsilon_1\alpha_{12})\rangle\otimes \varepsilon_2\alpha_{21}\varepsilon_1^2-\langle \alpha_{21}, \rho(\varepsilon_1^2\alpha_{12})\rangle\otimes \varepsilon_2\alpha_{21}\varepsilon_1 = -\varepsilon_2\alpha_{21}\varepsilon_1.
\]

\end{example}

Recall the definition of the cluster variables $x_n$ of rank 2 cluster algebras given in \Cref{subsection: rank 2}. We rewrite \Cref{cor: recursion of cc function general} in the current rank 2 situation.

\begin{corollary}\label{cor: recursion of cc function}
Let $M\in \replf H$ such that $M_{2,\mathrm{in}}$ is surjective. Then $M' \coloneqq F_2^+(M)\in \rep H^*$ is locally free and $M'_{2, \mathrm{out}}$ is injective such that $\rk M' = s_2(\rk M)$. We further have
\[
X_M(x_1, x_2) = X_{M'}(x_1', x_2')
\]
where $x_1' = x_1$ and $x_2' = x_2^{-1}(1 + x_1^b)$.
\end{corollary}

The following is the main result of this section.

\begin{theorem}\label{thm: main theorem rank 2}
The map
\[
M\mapsto X_M\in \mathbb Z[x_1^\pm, x_2^\pm]
\]
induces a bijection
\[
\{ M(n) \mid n\leq 0,\ n\geq 3 \} \longleftrightarrow \{x_n \mid n\leq 0,\ n\geq 3\}
\]
such that $X_{M(n)} = x_n$. In particular, each $x_n$ is distinct for $n\in \mathbb Z$.
\end{theorem}

\begin{proof}
We prove $X_{M(n+3)} = x_{n+3}$ and $X_{M(-n)} = x_{-n}$ for $n\geq 0$ by induction on $n$ using the recursion \Cref{cor: recursion of cc function}. For $n = 0$, we have $M(0) = E_2$ and $M(3) = E_1$, thus
\[
X_{M(0)} = x_2^{-1}(1 + x_1^b) = x_0 \quad \text{and} \quad X_{M(3)} = x_1^{-1}(1 + x_2^c) = x_3.
\]
Assume the statement is true for some $n\geq 0$. By the obvious symmetry between $H$ and $H^*$ by switching the orientation, we have that
\[
x_{n+4} = X_{M(n+3, H^*)}(x_3, x_2),
\]
where the notation $M(n+3, H^*)$ stresses that the module $M(n+3, H^*)$ is constructed for $H^*$ instead of $H$. We would like to apply \Cref{cor: recursion of cc function} to $M \coloneqq M(n+3, H^*)$ and the sink reflection functor
\[
F_1^+ \colon \rep H^* \rightarrow \rep H.
\]
The condition that $M_{1,\mathrm{in}}$ is surjective is guaranteed by \Cref{lemma: incoming map is surjective rank 2} (applied to the algebra $H^*$). Then by \Cref{cor: recursion of cc function}, we have
\[
X_{M(n+3, H^*)}(x_3, x_2) = X_{M(n+4, H)}(x_1, x_2)
\]
where $M(n+4, H) = F^+_1(M(n+3, H^*))$ and $x_1 = x_3^{-1}(1 + x_2^c)$. Immediately we obtain
\[
x_{n+4} = X_{M(n+4, H)}(x_1, x_2).
\]
The proof for $M(-n)$ for $n\geq 0$ uses a similar induction.

Now that $x_n = X_{M(n)}$ is a Laurent polynomial in $x_1$ and $x_2$, the unique minimal common denominator (up to a scalar) is easily seen to be $x^{\alpha(n)}$. Since the positive roots $\alpha(n)$ are distinct, so are the cluster variables $x_n$.
\end{proof}

\begin{remark}
To a pair $(b,c)\in \mathbb Z_{>0}^2$, one can also associate an algebra $H(b,c)$ defined as the path algebra $\mathbb CQ$ of the quiver $Q = \begin{tikzcd}
1\ar[loop left, "\varepsilon_1"] \ar[r, "\alpha"] &2 \ar[loop right, "\varepsilon_2"]
\end{tikzcd}$ modulo the relations $\varepsilon_1^c = 0$ and $\varepsilon_2^b = 0$.
When $b$ and $c$ are coprime, the algebra $H(b,c)$ is the same as $H(C,D,\Omega)$ for $D = \begin{bsmallmatrix}
c & 0\\
0 & b
\end{bsmallmatrix}$. However, when $b$ and $c$ are not coprime, the algebra $H$ is not included in the construction of \cite{geiss2017quivers}. We note that \Cref{thm: intro main theorem rank 2} can be easily adapted to using the algebras $H(b,c)$. In the case $b = c =2$, the algebra $H(b,c)$ coincides with a construction in \cite{labardini2022gentle} where the ordinary Caldero--Chapoton functions are shown to give cluster variables of a generalized cluster algebra.
\end{remark}

\section{Dynkin cases}\label{section: dynkin case}

The purpose of this section is to give a new proof of \Cref{thm: intro main theorem dynkin} (\Cref{thm: main theorem GLS Dynkin}). Let $C$ be of Dynkin type and $B = B(C, \Omega)$ the associated skew-symmetrizable matrix. Denote by $\Delta^+_C$ the set of positive roots associated to $C$.

\begin{definition}\label{def: admissible sequence}
A sequence $\mathbf i = (i_1, \dots, i_{k+1})\in I^{k+1}$ is called \emph{adapted to} an orientation $\Omega$ of $C$ (or $\Omega$-adapted) if
\begin{itemize}
    \item[] $i_1$ is a sink for $\Omega$,
    \item[] $i_2$ is a sink for $s_{i_1}(\Omega)$,
    \item[] $\quad \vdots$
    \item[] $i_{k}$ is a sink for $s_{i_{k-1}}\cdots s_{i_2}s_{i_1}(\Omega)$.
\end{itemize}
\end{definition}

The following lemma is well-known; see for example \cite[Chapter 3]{kirillov2016quiver}.

\begin{lemma}\label{lemma: sequence for a positive root}
Let $\beta$ be a positive root for $C$ and $\Omega$ be an orientation. Then there always exists a sequence $\mathbf i = (i_1, \dots, i_{k+1})$ adapted to $\Omega$ for $\beta$ such that
\begin{itemize}
    \item[] $s_{i_k}(\alpha_{i_{k+1}})\in \Delta^+_C$,
    \item[] $s_{i_{k-1}}s_{i_k}(\alpha_{i_{k+1}})\in \Delta^+_C$,
    \item[] $\quad \vdots$
    \item[] $\beta = s_{i_{1}}\cdots s_{i_{k-1}}s_{i_{k}}(\alpha_{i_{k+1}}) \in \Delta^+_C$.
\end{itemize}
\end{lemma}

It is clear that for such a sequence $\mathbf i$ in the above lemma, $i_{k+1}$ must not be equal to $i_{k}$. To any sequence $\mathbf i = (i_1, \dots, i_{k+1})$, consider the following path in $\mathbb T_n$
\[
t_0 \frac{i_1}{\quad\quad} t_1 \frac{i_2}{\quad\quad} \cdots \frac{i_k}{\quad\quad} t_k \frac{i_{k+1}}{\quad\quad} t_{k+1}.
\]
Recall the cluster mutations as introduced in \Cref{subsection: lf cc function} which generate cluster variables. Recursively performing cluster mutations from $t_0$ to $t_{k+1}$ we obtain an $n$-tuple $(x_{1;t_{k+1}},\dots, x_{n;t_{k+1}})$ of cluster variables associated to $t_{k+1}$. we denote $x_\mathbf i \coloneqq x_{i_{k+1};t_{k+1}}$.

Suppose that $\mathbf i$ is adapted to $\Omega$ for a positive root $\beta$ as in \Cref{lemma: sequence for a positive root}. Note that $i_\ell$ is a source for the orientation $\Omega_\ell \coloneqq s_{i_\ell}\cdots s_{i_1}(\Omega)$ for $\ell = 1, \dots, k$. Let $H_\ell \coloneqq H(C, D, \Omega_\ell)$ for $\ell = 1, \dots, k$ and $H_0 \coloneqq H$. We have source reflection functors
\[
F_{i_\ell}^- \colon \rep H_\ell \rightarrow \rep H_{\ell-1}.
\]
Define the module
\[
M_{\mathbf i} \coloneqq F_{i_1}^-\cdots F_{i_k}^- (E_{i_{k+1}})\in \rep H.
\]

\begin{theorem}[{\cite[Theorem 1.2]{geiss2018quivers}}]\label{thm: main theorem GLS Dynkin}
The map $M\mapsto X_M$ induces a bijection between isomorphism classes of locally free indecomposable rigid $H(C, D, \Omega)$-modules and the non-initial cluster variables of the cluster algebra $\mathcal A(B)$.
\end{theorem}

\begin{proof}
For $\mathbf i$ adapted to $\Omega$ for a positive root $\beta$, we show by induction on the length of $\mathbf i$ that
\begin{equation}\label{eq: star}
    \tag{$\star$} \text{$M_\mathbf i$ is locally free, indecomposable and rigid with rank vector $\beta$, and that $X_{M_\mathbf i} = x_\mathbf i$.}
\end{equation}
If $\mathbf i = (i)$ is of length one, then $M_\mathbf i = E_i$. Notice that $i$ is not necessarily a sink or source. As in \Cref{ex: cc function of E_k}, we have
\[
X_{E_i} = x_i^{-1}\left( \prod_{j = 1}^n x_j^{[b_{ji}]_+} + \prod_{j = 1}^n x_j^{[-b_{ji}]_+} \right) = x_{(i)}.
\]

Assume that (\ref{eq: star}) is true for any $\mathbf i$ of length no greater than $k\in \mathbb N$. Let $\mathbf i = (i_\ell)_{\ell=1}^{k+1}$ and $\mathbf i'$ be the sequence
\[
(i_2, i_3, \dots, i_{k+1})\in I^{k},
\]
which is adapted to the orientation $s_{i_1}(\Omega)$. By assumption, the module
\[
M_{\mathbf i'} \coloneqq F_{i_2}^-\cdots F_{i_k}^-(E_{i_{k+1}})\in \rep s_{i_1}(H)
\]
is locally free, rigid and indecomposable with rank vector $\beta' \coloneqq s_{i_2}\dots s_{i_k}(\alpha_{i_{k+1}})$ and that $X_{M_{\mathbf i'}} = x_{\mathbf i'}\in \mathcal A(\mu_{i_1}(B))$. Since $s_{i_1}(\beta') = \beta$ and $\beta' \in \Delta^+(C)$, the positive root $\beta'$ cannot be a positive multiple of $\alpha_{i_1}$. Thus the indecomposable module $M_{\mathbf i'}$ does not have any direct summand isomorphic to $E_{i_1}$. By \Cref{lemma: locally free and rigid implies surjective}, the map $(M_{\mathbf i'})_{i_1,\mathrm{out}}$ is injective. By \Cref{prop: reflection preserve locally free and rigid}, we have that $M_{\mathbf i} = F_{i_1}^-(M_{\mathbf i'})$ is locally free, rigid and indecomposable with rank vector $\beta = s_{i_1}(\beta')$.

Applying \Cref{cor: recursion of cc function general} to $M_\mathbf i\in \rep H$ and $M_{\mathbf i'} \cong F_{i_1}^+(M_\mathbf i) \in \rep s_{i_1}(H)$, we have
\[
X_{M_\mathbf i}(x_1,\dots,x_{i_1},\dots,x_n) = X_{M_{\mathbf i'}}(x_1, \dots, x_{i_1}', \dots, x_n),
\]
where 
\[
x_{i_1}' = x_{i_1}^{-1}\left(1 + \prod_{j\in \Omega(i_1,-)} x_j^{-b_{ji_1}} \right).
\]
Notice that
\[
X_{M_{\mathbf i'}}(x_1, \dots,x_{i_1}',\dots, x_n) = x_{\mathbf i'}(x_1, \dots,x_{i_1}',\dots x_n) = x_\mathbf i(x_1, \dots, x_{i_1}, \dots, x_n)\in \mathcal A(B).
\]
Hence $X_{M_\mathbf i} = x_{\mathbf i}\in \mathbb Z[x_1^\pm, \dots, x_n^\pm]$, which completes the induction and proves (\ref{eq: star}).

By \cite[Theorem 1.3]{geiss2017quivers}, the module ${M_\mathbf i}$ constructed from $\mathbf i$ only depends on the positive root $\beta$ and the (thus well-defined) map $\beta\mapsto {M_\mathbf i}$ induces a bijection from $\Delta^+_C$ to locally free indecomposable rigid $H$-modules (up to isomorphism). Thus the formula $x_\mathbf i =  X_{M_\mathbf i}$ implies that the cluster variable $x(\beta)\coloneqq x_\mathbf i$ also only depends on $\beta$. In view of \Cref{rmk: express cc function by f polynomial}, each $x(\beta)$ has $\mathbf d$-vector $\beta$. By \cite{fomin2003cluster}, these $x(\beta)$ are exactly the non-initial cluster variables of $\mathcal A(B)$, hence the desired bijection.
\end{proof}

\section{Beyond Dynkin and rank 2 cases}\label{section: beyond}
For $(C, D, \Omega)$ which is neither of Dynkin type nor in rank 2, in general we will not be able to reach all locally free indecomposable rigid modules by reflections. In this section, we prove locally free Caldero--Chapoton formulas for cluster variables that can be obtained by \emph{almost} sink and source mutations. In particular, any cluster variable on the bipartite belt \cite{fomin2007cluster} can be obtained this way.

\begin{definition}[cf. {\cite{rupel2011quantum}}]
A sequence $\mathbf i = (i_1, \dots, i_{k+1})\in I^{k+1}$ is called \emph{admissible} to an orientation $\Omega$ of $C$ if
\begin{itemize}
    \item[] $i_1$ is a sink or source for $\Omega$,
    \item[] $i_2$ is a sink or source for $s_{i_1}(\Omega)$,
    \item[] $\quad \vdots$
    \item[] $i_{k}$ is a sink or source for $s_{i_{k-1}}\cdots s_{i_2}s_{i_1}(\Omega)$.
\end{itemize}
\end{definition}

Let $B = B(C, \Omega)$ and $\mathcal A(B)$ be the (coefficient-free) cluster algebra associated to $B$. As defined in \Cref{section: dynkin case}, for an arbitrary sequence $\mathbf i$, there is the cluster variable $x_\mathbf i\in \mathcal A(B)$ by successive cluster mutations.

\begin{definition}
We say that the cluster variable $x_\mathbf i\in \mathcal A(B)$ corresponding to a sequence $\mathbf i = (i_\ell)_{\ell}$ is obtained by \emph{almost sink or source mutations} if $\mathbf i$ is admissible to $\Omega$.
\end{definition}

\begin{remark}
We note that by definition the last index $i_{k+1}$ can be arbitrary in $I$. It is the only step in the mutation sequence $(\mu_{i_1}, \dots, \mu_{i_{k+1}})$ that may not be at a sink or source, thus the term \emph{almost}.
\end{remark}

The following is our main result in this section.

\begin{theorem}\label{thm: main theorem general}
For any admissible sequence $\mathbf i$, either the cluster variable $x_\mathbf i$ is an initial one or there is a locally free indecomposable rigid $H(C,D,\Omega)$-module $M_\mathbf i$ such that
\[
X_{M_\mathbf i}(x_1, \dots, x_n) = x_\mathbf i.
\]
Moreover, the module $M_\mathbf i$ is uniquely determined (up to isomorphism) by $x_\mathbf i$.
\end{theorem}

\begin{proof}
We slightly modify the functors $F_i^\pm$ to define the operations
\[
    F_i^\pm \colon \rep H\cup \{x_1, \dots, x_n\} \rightarrow \rep s_i(H)\cup \{x_1, \dots, x_n\}
\]
such that
\[
F_{i}^\pm(M) \coloneqq \begin{cases}
F_i^\pm(M) & \text{if $M\in \rep H$ not isomorphic to $E_i$},\\
x_i & \text{if $M \cong E_i$},\\
E_i & \text{if $M = x_i$},\\
x_j & \text{if $M = x_j$ for $j\neq i$}.
\end{cases}
\]
For an admissible sequence $\mathbf i$, let
\[
M_\mathbf i \coloneqq F_{i_1}^\pm F_{i_2}^\pm \cdots F_{i_{k}}^\pm (E_{i_{k+1}}) \in \rep H\cup \{x_1, \dots, x_n\}
\]
where each sign is chosen on whether $i_\ell$ is a sink or source of $s_{i_{\ell-1}}\cdots s_{i_1}(\Omega)$.
We define $X_{M_\mathbf i}$ as in \Cref{def: cc function general} if $M_\mathbf i$ is a locally free $H$-module or $X_{M_\mathbf i} \coloneqq x_i$ if $M_\mathbf i = x_i$ for some $i\in \{1,\dots, n\}$. We next show by induction on the length of $\mathbf i$ that
\begin{enumerate}
    \item[$\bullet$] if $M_\mathbf i$ is indeed a module, then it must be locally free, indecomposable and rigid, and
    \item[$\bullet$] $X_{M_\mathbf i} = x_\mathbf i$.
\end{enumerate}

The induction is a slight modification of the proof of \Cref{thm: main theorem GLS Dynkin}. For $\mathbf i = (i)$ of length one, $M_{(i)} = E_i$ and the statement is clearly true. Assume that the statement is true for $\mathbf i$ of length no greater than $k$. Let $\mathbf i = (i_\ell)_\ell$ be of length $k+1$ admissible to $\Omega$ and $\mathbf i' = (i_2, \dots, i_{k+1})$, which is admissible to $s_{i_1}(\Omega)$. By assumption, either
\begin{enumerate}
    \item $M_{\mathbf i'}\in \rep s_{i_1}(H)$ is locally free indecomposable rigid or
    \item $M_{\mathbf i'} = x_i$ for some $i\in I$.
\end{enumerate}
In both cases, we have by assumption that $X_{M_{\mathbf i'}} = x_{\mathbf i'} \in \mathcal A(\mu_{i_1}(B))$.

In case (1), there are three sub-cases
\begin{enumerate}
    \item [(1.1)] $M_{\mathbf i'} \cong E_{i_1}$;
    \item [(1.2)] $i_1$ is a sink of $s_{i_1}(\Omega)$ and $(M_{\mathbf i'})_{i_1, \mathrm{in}}$ is surjective by \Cref{lemma: locally free and rigid implies surjective};
    \item [(1.3)] $i_1$ is a source of $s_{i_1}(\Omega)$ and $(M_{\mathbf i'})_{i_1, \mathrm{out}}$ is injective by \Cref{lemma: locally free and rigid implies surjective}.
\end{enumerate}
In case (2), there are two sub-cases
\begin{enumerate}
    \item [(2.1)] $i = i_1$ and thus $F_{i_1}^\pm(M_\mathbf i) = E_i$;
    \item [(2.2)] $i \neq i_1$ and thus $F_{i_1}^\pm (M_\mathbf i) = x_i$.
\end{enumerate}
By \Cref{cor: recursion of cc function general}, it is easy to check that in each of the sub-cases we have
\[
X_{M_\mathbf i}(x_1, \dots, x_{i_1},\dots, x_n) = X_{M_{\mathbf i'}} (x_1', \dots, x_{i_1}', \dots, x_n')
\]
where $x_i' = x_i$ if $i\neq i_1$ and $x_{i_1}' = x_{i_1}^{-1}\left(\prod_{j} x_j^{[b_{ji_1}]_+} + \prod_{j} x_j^{[-b_{ji_1}]_+}\right)$ and by \Cref{prop: reflection preserve locally free and rigid} if $M_\mathbf i = F_{i_1}^\pm (M_{\mathbf i'})$ is a module, it is locally free, indecomposable and rigid. In all cases, we have
\[
X_{M_\mathbf i}(x_1, \dots, x_{i_1}, \dots, x_n) = x_{\mathbf i'}(x_1', \dots, x_{i_1}', \dots, x_n') = x_\mathbf i\in \mathcal A(B),
\]
which finishes the induction.

By the Caldero--Chapoton formula, the rank vector $\rk M$ is just the $\mathbf d$-vector of $x_\mathbf i$, thus only depending on $x_\mathbf i$; see \Cref{rmk: express cc function by f polynomial}. It is shown in \cite{geiss2018rigid} that any locally free indecomposable rigid module is determined by its rank vector (which in particular is a real Schur root of $(C, \Omega)$). Thus $M_\mathbf i$ is uniquely determined by $x_\mathbf i$.
\end{proof}

\bibliographystyle{amsalpha-fi-arxlast}
\bibliography{reference.bib}

\providecommand{\bysame}{\leavevmode\hbox to3em{\hrulefill}\thinspace}
\providecommand{\MR}{\relax\ifhmode\unskip\space\fi MR }
\providecommand{\MRhref}[2]{%
  \href{http://www.ams.org/mathscinet-getitem?mr=#1}{#2}
}
\providecommand{\href}[2]{#2}
\begin{thebibliography}{DWZ10}

\bibitem[BGP73]{bernstein1973coxeter}
I.~N. Bern\v{s}te\u{\i}n, I.~M. Gel'fand, and V.~A. Ponomarev, \emph{Coxeter
  functors, and {G}abriel's theorem}, Uspehi Mat. Nauk \textbf{28} (1973),
  no.~2(170), 19--33.

\bibitem[BLA21]{bautista2021potentials}
R.~Bautista and D.~L\'{o}pez-Aguayo, \emph{Potentials for some tensor
  algebras}, J. Algebra \textbf{573} (2021), 177--269.

\bibitem[CC06]{caldero2006cluster}
P.~Caldero and F.~Chapoton, \emph{Cluster algebras as {H}all algebras of quiver
  representations}, Comment. Math. Helv. \textbf{81} (2006), no.~3, 595--616.

\bibitem[CK06]{caldero2006triangulated}
P.~Caldero and B.~Keller, \emph{From triangulated categories to cluster
  algebras. {II}}, Ann. Sci. \'{E}cole Norm. Sup. (4) \textbf{39} (2006),
  no.~6, 983--1009.

\bibitem[CZ06]{caldero2006laurent}
P.~Caldero and A.~Zelevinsky, \emph{Laurent expansions in cluster algebras via
  quiver representations}, Mosc. Math. J. \textbf{6} (2006), no.~3, 411--429,
  587.

\bibitem[Dem]{demonet2010mutations}
L.~Demonet, \emph{Mutations of group species with potentials and their
  representations. {A}pplications to cluster algebras}, arXiv:1003.5078.

\bibitem[Dem11]{demonet2011categorification}
L.~Demonet, \emph{Categorification of skew-symmetrizable cluster algebras},
  Algebr. Represent. Theory \textbf{14} (2011), no.~6, 1087--1162.

\bibitem[DWZ08]{derksen2008quivers}
H.~Derksen, J.~Weyman, and A.~Zelevinsky, \emph{Quivers with potentials and
  their representations. {I}. {M}utations}, Selecta Math. (N.S.) \textbf{14}
  (2008), no.~1, 59--119.

\bibitem[DWZ10]{derksen2010quivers}
H.~Derksen, J.~Weyman, and A.~Zelevinsky, \emph{Quivers with potentials and
  their representations {II}: applications to cluster algebras}, J. Amer. Math.
  Soc. \textbf{23} (2010), no.~3, 749--790.

\bibitem[FGL20]{fu2020cluster}
C.~Fu, S.~Geng, and P.~Liu, \emph{Cluster algebras arising from cluster tubes
  {II}: {T}he {C}aldero-{C}hapoton map}, J. Algebra \textbf{544} (2020),
  228--261.

\bibitem[FZ02]{fomin2002cluster}
S.~Fomin and A.~Zelevinsky, \emph{Cluster algebras {I}: {F}oundations}, J.
  Amer. Math. Soc. \textbf{15} (2002), no.~2, 497--529.

\bibitem[FZ03]{fomin2003cluster}
S.~Fomin and A.~Zelevinsky, \emph{Cluster algebras. {II}. {F}inite type
  classification}, Invent. Math. \textbf{154} (2003), no.~1, 63--121.

\bibitem[FZ07]{fomin2007cluster}
S.~Fomin and A.~Zelevinsky, \emph{Cluster algebras {IV}: {C}oefficients},
  Compos. Math. \textbf{143} (2007), no.~1, 112--164.

\bibitem[Gab72]{gabriel1972unzerlegbare}
P.~Gabriel, \emph{Unzerlegbare {D}arstellungen. {I}}, Manuscripta Math.
  \textbf{6} (1972), 71--103; correction, ibid. 6 (1972), 309.

\bibitem[Gei18]{geiss2018icm}
C.~Gei\ss, \emph{Quivers with relations for symmetrizable {C}artan matrices and
  algebraic {L}ie theory}, Proceedings of the {I}nternational {C}ongress of
  {M}athematicians---{R}io de {J}aneiro 2018. {V}ol. {II}. {I}nvited lectures,
  World Sci. Publ., Hackensack, NJ, 2018, pp.~99--124.

\bibitem[GLF17]{Geuenich-LF1}
J.~Geuenich and D.~Labardini-Fragoso, \emph{Species with potential arising from
  surfaces with orbifold points of order 2, {P}art {I}: one choice of weights},
  Math. Z. \textbf{286} (2017), no.~3, 1065--1143.

\bibitem[GLF20]{Geuenich-LF2}
J.~Geuenich and D.~Labardini-Fragoso, \emph{Species with potential arising from
  surfaces with orbifold points of order 2, {P}art {II}: {A}rbitrary weights},
  Int. Math. Res. Not. IMRN (2020), no.~12, 3649--3752.

\bibitem[GLS11]{geiss2011kac}
C.~Gei\ss, B.~Leclerc, and J.~Schr\"{o}er, \emph{Kac-{M}oody groups and cluster
  algebras}, Adv. Math. \textbf{228} (2011), no.~1, 329--433.

\bibitem[GLS16]{geiss2016quiversIII}
C.~Gei\ss, B.~Leclerc, and J.~Schr\"{o}er, \emph{Quivers with relations for
  symmetrizable {C}artan matrices {III}: {C}onvolution algebras}, Represent.
  Theory \textbf{20} (2016), 375--413.

\bibitem[GLS17]{geiss2017quivers}
C.~Gei\ss, B.~Leclerc, and J.~Schr\"{o}er, \emph{Quivers with relations for
  symmetrizable {C}artan matrices {I}: {F}oundations}, Invent. Math.
  \textbf{209} (2017), no.~1, 61--158.

\bibitem[GLS18]{geiss2018quivers}
C.~Gei\ss, B.~Leclerc, and J.~Schr\"{o}er, \emph{Quivers with relations for
  symmetrizable {C}artan matrices {V}: {C}aldero-{C}hapoton formulas}, Proc.
  Lond. Math. Soc. (3) \textbf{117} (2018), no.~1, 125--148.

\bibitem[GLS20]{geiss2018rigid}
C.~Gei\ss, B.~Leclerc, and J.~Schr\"{o}er, \emph{Rigid modules and {S}chur
  roots}, Math. Z. \textbf{295} (2020), no.~3-4, 1245--1277.

\bibitem[Kir16]{kirillov2016quiver}
A.~Kirillov, Jr., \emph{Quiver representations and quiver varieties}, Graduate
  Studies in Mathematics, vol. 174, American Mathematical Society, Providence,
  RI, 2016.

\bibitem[LA19]{lopez2019species}
D.~L\'{o}pez-Aguayo, \emph{A note on species realizations and nondegeneracy of
  potentials}, J. Algebra Appl. \textbf{18} (2019), no.~2, 1950024, 9.

\bibitem[LFMa]{labardini2022gentle}
D.~Labardini-Fragoso and L.~Mou, \emph{Gentle algebras arising from surfaces
  with orbifold points of order 3, {P}art {I}: scattering diagrams},
  \mbox{arXiv:2203.11563}.

\bibitem[LFMb]{LMpart2}
D.~Labardini-Fragoso and L.~Mou, \emph{Gentle algebras arising from surfaces
  with orbifold points of order 3, {P}art {II}: laminations and generic bases},
  in preparation.

\bibitem[LFZ16]{labardini2016strongly}
D.~Labardini-Fragoso and A.~Zelevinsky, \emph{Strongly primitive species with
  potentials {I}: mutations}, Bol. Soc. Mat. Mex. (3) \textbf{22} (2016),
  no.~1, 47--115.

\bibitem[Rup11]{rupel2011quantum}
D.~Rupel, \emph{On a quantum analog of the {C}aldero-{C}hapoton formula}, Int.
  Math. Res. Not. IMRN (2011), no.~14, 3207--3236.

\bibitem[Rup15]{rupel2015quantum}
D.~Rupel, \emph{Quantum cluster characters for valued quivers}, Trans. Amer.
  Math. Soc. \textbf{367} (2015), no.~10, 7061--7102.

\bibitem[SZ04]{sherman2004positivity}
P.~Sherman and A.~Zelevinsky, \emph{Positivity and canonical bases in rank 2
  cluster algebras of finite and affine types}, Mosc. Math. J. \textbf{4}
  (2004), no.~4, 947--974, 982.

\bibitem[YZ08]{yang2008cluster}
S.-W. Yang and A.~Zelevinsky, \emph{Cluster algebras of finite type via
  {C}oxeter elements and principal minors}, Transform. Groups \textbf{13}
  (2008), no.~3-4, 855--895.

\end{thebibliography}

\end{document}